\documentclass[a4paper]{amsart}
\usepackage{hyperref}
\usepackage[utf8]{inputenc}
\usepackage{txfonts, amsmath,amstext,amsthm,amscd,amsopn,verbatim,amssymb, amsfonts}
\usepackage{fullpage}
\usepackage{todonotes}
\usepackage{mathtools}

\usepackage[bbgreekl]{mathbbol}
\usepackage{enumerate}

\usepackage{dsfont}

\usepackage{float}
\restylefloat{table}

\usepackage{tikz}
\usepackage{tikz-cd}
\usetikzlibrary{matrix}
\usetikzlibrary{shapes}
\usetikzlibrary{arrows}
\usetikzlibrary{calc,3d}
\usetikzlibrary{decorations,decorations.pathmorphing,decorations.pathreplacing,decorations.markings}
\usetikzlibrary{through}

\tikzset{snakeit/.style={decorate, decoration={snake, amplitude=.2mm,segment length=1mm}}}

\tikzset{ext/.style={circle, draw,inner sep=1pt}, int/.style={circle,draw,fill,inner sep=2pt},nil/.style={inner sep=1pt}}
\tikzset{cy/.style={circle,draw,fill,inner sep=2pt},scy/.style={circle,draw,inner sep=2pt},scyx/.style={draw,cross out,inner sep=2pt},scyt/.style={draw,regular polygon,regular polygon sides=3,inner sep=0.95pt}}
\tikzset{exte/.style={circle, draw,inner sep=3pt},inte/.style={circle,draw,fill,inner sep=3pt}}
\tikzset{diagram/.style={matrix of math nodes, row sep=3em, column sep=2.5em, text height=1.5ex, text depth=0.25ex}}
\tikzset{diagram2/.style={matrix of math nodes, row sep=0.5em, column sep=0.5em, text height=1.5ex, text depth=0.25ex}}
\tikzset{rowcolsep/.style={column sep=.2cm, row sep=.1cm}}
\tikzset{cross/.style={cross out,draw}}
\tikzset{every loop/.style={draw}}

\tikzset{
  crossed/.style={
    decoration={markings,mark=at position .5 with {\arrow{|}}},
    postaction={decorate},
    shorten >=0.4pt}}

\tikzset{every picture/.style={baseline=-.65ex} }

\theoremstyle{plain}
  \newtheorem{thm}{Theorem}
  
  \newtheorem{prop}[thm]{Proposition}
  
  \newtheorem{lemdef}[thm]{Lemma/Definition}
  \newtheorem{cor}[thm]{Corollary}
  
  \newtheorem{lemma}[thm]{Lemma}
\theoremstyle{definition}
  \newtheorem{ex}{Example}

\newcommand{\alg}[1]{\mathfrak{{#1}}}

\newcommand{\Hom}{\mathrm{Hom}}

\newcommand{\Z}{{\mathbb{Z}}}

\newcommand{\Graphs}{{\mathsf{Graphs}}}

\newcommand{\mF}{\mathcal{F}}

\newcommand{\FreeLie}{\mathrm{FreeLie}}
\newcommand{\FreeCom}{S}

\newcommand{\Com}{\mathsf{Com}}

\newcommand{\bpm}{\begin{pmatrix}}
\newcommand{\epm}{\end{pmatrix}}

\newcommand{\Aut}{\mathrm{Aut}}
\newcommand{\GC}{\mathsf{GC}}

\newcommand{\G}{\mathsf{G}}

\DeclareMathOperator{\gr}{gr}

\newcommand{\gra}{\mathit{gra}}

\newcommand{\Q}{\mathbb{Q}}

\newcommand{\fg}{\mathfrak{g}}

\renewcommand{\Bar}{\mathtt{B}}

\newcommand{\osp}{\mathfrak{osp}}

\newcommand{\ft}{\mathfrak{t}}

\DeclareMathOperator{\tru}{\mathrm{tr}}

\newcommand{\fG}{\mathsf{fG}}
\newcommand{\GCex}{\GC^\ex}

\renewcommand{\ex}{{\mathrm{ex}}}

\newcommand{\cut}{\mathrm{cut}}
\newcommand{\GL}{\mathrm{GL}}

\author{Simon Brun}
\address{Department of Mathematics\\ ETH Zurich\\ R\"amistrasse 101 \\ 8092 Zurich, Switzerland}
\email{simonbrun@math.ethz.ch}

\author{Thomas Willwacher}
\address{Department of Mathematics\\ ETH Zurich\\ R\"amistrasse 101 \\ 8092 Zurich, Switzerland}
\email{thomas.willwacher@math.ethz.ch}

\begin{document}
\title{Stable cohomology of graph complexes in odd dimensions}

\thanks{
  S.B. and T.W. have been supported by the ERC starting grant 678156 GRAPHCPX, and the NCCR SwissMAP, funded by the Swiss National Science Foundation.
}

\begin{abstract}
We study graph complexes related to configuration spaces and diffeomorphism groups of highly connected manifolds of odd dimension. In particular we compute the cohomology in the "high genus" limit.
This paper is a continuation of previous work by Felder, Naef and the second author in which the even dimensional case is studied \cite{FNW}.
\end{abstract}

\maketitle

\tableofcontents

\section{Introduction}
Let $m$ be a fixed positive integer.
In this paper we study the graph complex $\GC_{(g),1}$ that arises in connection to the configuration spaces and diffeomorphism groups of the $2m+1$-dimensional manifold $W_{g,1}$:
\[
   W_{g,1} = W_g \setminus D^{2m+1} \quad \text{with} \quad W_g = \#^g S^{m}\times S^{m+1}.
\]

Elements of $\GC_{(g),1}$ are $\Q$-linear series of isomorphism classes of connected, at least trivalent graphs whose vertices are decorated by elements of the reduced homology $\bar H_\bullet(W_{g,1})$:
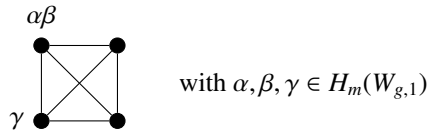
\[
\begin{tikzpicture}[yshift=-.5cm]
\node[int,label=180:{$\gamma$}] (v1) at (0,0) {};
\node[int] (v2) at (1,0) {};
\node[int,label=90:{$\alpha\beta$}] (v3) at (0,1) {};
\node[int] (v4) at (1,1) {};
\draw (v1) edge (v3) edge (v2) edge (v4) 
    (v4) edge (v2) edge (v3) 
    (v3) edge (v2);
\end{tikzpicture}
\quad\quad \text{with $\alpha,\beta,\gamma \in H_m(W_{g,1})$}
\]

Note that the graph complex $\GC_{(g),1}$ depends on the manifold dimension $n=2m+1$ which we will consider implicit in the notation. Especially note that this complex differs from the one studied in \cite{FNW}, with identical notation. Here we consider the case where the manifold dimension $n$ is odd compared to the even dimensional case studied in \cite{FNW}. 

The differential on this complex has two terms, $\delta=\delta_{split}+\delta_{glue}$. The piece $\delta_{split}$ is defined by summing over vertices, and splitting the vertex:
  \begin{align}\label{equ:deltasplit}
    \delta_{split} \Gamma &= \sum_{v \text{ vertex} }  \pm 
    \Gamma\text{ split $v$} 
    &
    \begin{tikzpicture}[baseline=-.65ex]
    \node[int] (v) at (0,0) {};
    \draw (v) edge +(-.3,-.3)  edge +(-.3,0) edge +(-.3,.3) edge +(.3,-.3)  edge +(.3,0) edge +(.3,.3);
    \end{tikzpicture}
    &\mapsto
    \sum
    \begin{tikzpicture}[baseline=-.65ex]
    \node[int] (v) at (0,0) {};
    \node[int] (w) at (0.5,0) {};
    \draw (v) edge (w) (v) edge +(-.3,-.3)  edge +(-.3,0) edge +(-.3,.3)
     (w) edge +(.3,-.3)  edge +(.3,0) edge +(.3,.3);
    \end{tikzpicture}
  \end{align}
The piece $\delta_{glue}\Gamma$ is defined on a graph $\Gamma$ by summing over all pairs $(\alpha,\beta)$ of $\bar H_\bullet(W_{g,1})$-decorations in the graph $\Gamma$, replacing the pair of decorations by an edge, and multiplying the graph with the numeric prefactor $\langle \alpha,\beta\rangle$, using the canonical skew-symmetric pairing $\langle -,-\rangle: \bar H_\bullet(W_{g,1})\times \bar H_\bullet(W_{g,1}) \to \Q$.
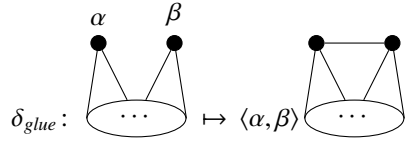
\[
	\delta_{glue} \colon \, 
\begin{tikzpicture}
	\node[int,label=90:{$\alpha$}] (v) at (-.5,1) {};
	\node[int,label=90:{$\beta$}] (w) at (.5,1) {};	
	\node[draw, ellipse, minimum width=1.3cm] at (0,0) (c) {$\cdots$};
	\draw (v) edge (c.west) edge (c) 
	(w) edge (c) edge (c.east);
\end{tikzpicture}
\,
\mapsto 
\,
\langle\alpha,\beta \rangle
\,
\begin{tikzpicture}
	\node[int] (v) at (-.5,1) {};
	\node[int] (w) at (.5,1) {};	
	\node[draw, ellipse, minimum width=1.3cm] at (0,0) (c) {$\cdots$};
	\draw (v) edge (c.west) edge (c) 
	(w) edge (c) edge (c.east)
	(v) edge (w);
\end{tikzpicture}
\]

We refer to section \ref{sec:GCs} below for more precise definitions, including signs, prefactors and degrees. We note that this complex depends on the chosen integer $m$, although this dependence is kept implicit in the notation.

The complex is in fact a dg Lie algebra, with the Lie brackets defined similarly to $\delta_{glue}$ above, just operating on a pair of decorations on two distinct graphs. 

The above graph complex $\GC_{(g),1}$ carries a natural grading  by \emph{weight}, with the weight of a graph with $e$ edges, $v$ vertices and $D$ decorations in $\bar H_\bullet(W_{g,1})$ defined to be the number
\[
W = 2(e-v) + D.
\]
This positive integer valued quantity is preserved by the differential and the Lie bracket. In particular our graph complex splits into a direct product of finite dimensional subcomplexes according to weight.
Furthermore, the differential and dg Lie structure also preserve the difference
\[
H := \text{\#decorations in $H_m\left(W_{g,1}\right)$} - \text{\#decorations in $H_{m+1}\left(W_{g,1}\right)$}\in \Z,
\]
that we shall call \emph{imbalance}.
Together the numbers $(W,H)$ equip our graph complex with a $\Z\times \Z$-grading.
We shall denote the graded piece of the complex or cohomology of given bidegree $(W,H)\in \Z\times \Z$ by the prefix $\gr^{(W,H)} (\cdots)$. We also denote the part of weight $W$ (and any imbalance) by $\gr^W(\cdots)$.

The first part of this paper studies the cohomology of our dg Lie algebra $\GC_{(g),1}$ for large $g$. To this end we have the following vanishing result.

\begin{thm}\label{thm:main cohom GC vanishing} 
  Let 
  $k_{crit}^L(W,H):=-\frac W2 (2m-1)+\frac H2$.
\begin{enumerate}[(i)]
  \item For all $g\geq 0$, $W\geq 1$, $H\in \Z$ and $k<k_{crit}^L(W,H)$ we have that 
  \[
    \gr^{(W,H)} H^{k}\left(\GC_{(g),1}\right)
    =0.
  \] 
  \item For all $W\geq 1$, $g\geq W+2$, $H\in \Z$ and $k>k_{crit}^L(W,H)$ we have that
  \[
    \gr^{(W,H)} H^{k}\left(\GC_{(g),1}\right) = 0.
  \] 
\end{enumerate}
\end{thm}
In other words, for $g\geq W+2$, the cohomology of the weight $W$ and imbalance $H$ part of $\GC_{(g),1}$ becomes concentrated in the \emph{critical degree} $k_{crit}^L(W,H)$.

A similar result can be obtained for the Chevalley-Eilenberg cohomology $H_{CE}(\GC_{(g),1})$ of the Lie algebra.
We define the Chevalley-Eilenberg complex of $\GC_{(g),1}$ as the cobar construction of the graded dual 
\[
  C_{CE}\left(\GC_{(g),1}\right) = \Bar^c \left(\left(\GC_{(g),1}\right)^c\right).
\]
Then $C_{CE}(\GC_{(g),1})$ is a differential graded commutative algebra. It inherits the $\Z\times Z$-grading by weight and imbalance from $\GC_{(g),1}$.
We denote by $\gr^{(W,H)} H_{CE}(\GC_{(g),1})$ the weight $W$ and imbalance $H$ piece of the Chevalley-Eilenberg cohomology.

\begin{thm}\label{thm:main CE all}
  Let $k_{crit}^C(W,H):= \frac W2 (2m+1) -\frac H 2$.
For any $g\geq 3W$, $H\in \Z$ and $k\neq k_{crit}^C(W,H)$ we have that 
\[
  \gr^{(W,H)} H_{CE}^k\left(\GC_{(g),1}\right) = 0.
\]
\end{thm}
Finally, we compute the non-vanishing cohomologies in the critical degrees for high genera. 
This can be achieved by using standard results of Koszul duality theory.
First let us define the graded Lie algebra
\[
\ft_{(g),1} = \FreeLie\left(\gr^1H\left(\GC_{(g),1}\right)\right) / \langle R_{(g),1} \rangle 
\]
generated by the weight 1 part of  $H(\GC_{(g),1})$, with the quadratic relations $R_{(g),1}$ defined as 
\[
  R_{(g),1} \cong \gr^2 H_{CE}\left(\GC_{(g),1}\right)^*.
\]
Furthermore, we also define the graded commutative algebra
\[
  A_{(g),1} = \FreeCom\left(\gr^1H_{CE}\left(\GC_{(g),1}\right)\right) / \langle S_{(g),1} \rangle 
\]
generated by the weight 1 part of  $H_{CE}\left(\GC_{(g),1}\right)$, with the quadratic relations 
\[
  S_{(g),1} \cong \gr^2 H\left(\GC_{(g),1}\right)^*.
\]
From (more or less) standard results in Koszul duality theory, see Proposition \ref{prop:cohom} below, we can then deduce the following theorem.
\begin{thm}\label{thm:main Koszul}
There is a morphism of graded Lie algebras
\[
  \ft_{(g),1} \to H(\GC_{(g),1})
\]
and a morphism of dg commutative algebras
\[
  A_{(g),1} \to H_{CE}\left(\GC_{(g),1}\right)
\]
that respect the gradings by weight and imbalance, and that are isomorphisms on the parts of weight $W\leq \frac g3$.

Furthermore, $\ft_{(g),1}$ and $A_{(g),1}$ form a Koszul pair in the same weight range. 
\end{thm}
We also note that Proposition \ref{prop:cohom} in fact asserts that $\GC_{(g),1}$ and $C_{CE}\left(\GC_{(g),1}\right)$ are formal in the above range of weights, so that we understand the homotopy type of those dg Lie (resp. dg commutative) algebras in that range.

\smallskip

One has a natural action of the general linear group $\GL_g$ on all the objects considered above. More precisely, if we denote by $V_{g}=\Q^g$ the $g$-dimensional defining representation of $\GL_g$, then we have that $H^m(W_{g,1})\cong V_g$ and $H^{m+1}(W_{g,1})\cong V_g^*$, as $\GL_g$-representations.
Hence, $\GL_g$ acts on $\GC_{(g),1}$ respecting the $\Z\times\Z$-grading. Therefore, $\GL_g$ also acts on the cohomology of $\GC_{(g),1}$, its Lie algebra cohomology and the objects $\ft_{(g),1}$ and $A_{(g),1}$ derived from it.
In section \ref{sec:representations} the generators and relations of $\ft_{(g),1}$ and $A_{(g),1}$ will be computed explicitly as irreducible representations of $\GL_g$, so that we obtain a complete description of the cohomology, and Lie algebra cohomology of $\GC_{(g),1}$ in the range $W\leq \frac g3$.

Moreover, $\GC_{(g),1}$  may naturally be extended by the nilpotent, negatively graded Lie algebra $\osp^{nil}_{g,1}$ of endomorphisms of $H(W_{g,1})$ that respect the pairing pairing $\langle -,-\rangle: \bar H_\bullet(W_{g,1})\times \bar H_\bullet(W_{g,1}) \to \Q$, see section \ref{ospnil} for details.
(In particular, $\osp_{g,1}^{nil}$ is concentrated in cohomological degree $-1$.)
We then define the extended dg Lie algebra
\[
  \GCex_{(g),1} := \left(\osp_{g,1}^{nil}\ltimes \GC_{(g),1}, \delta\right).
\]

We extend the $\Z\times\Z$-bigrading to $\GCex_{(g),1}$ by declaring the part $\osp_{g,1}^{nil}$ to be concentrated in weight $W=0$ and imbalance $H=-2$, compatible with the dg Lie structure.
In particular, $\osp_{g,1}^{nil}$ acts on $\GCex_{(g),1}$ respecting the weight grading and 
\[
H\left(\GCex_{(g),1}\right) = \osp_{g,1}^{nil}\ltimes H(\GC_{(g),1}).
\]
We hence obtain the following Corollary of Theorems \ref{thm:main cohom GC vanishing} and \ref{thm:main Koszul}.

\begin{cor}
The cohomology $\gr^{(W,H)}H^k\left(\GCex_{(g),1}\right)$ vanishes if either $k<k^L_{crit}(W,H)$ or $k>k^L_{crit}(W,H)$ and $g\geq W+2$.
Furthermore, the morphism of $\Z\times\Z$-graded dg Lie algebras 
\[
	\osp_{g,1}^{nil}\ltimes\ft_{(g),1} \to H(\GCex_{(g),1})
\]
is an isomorphism on cohomology in weights $W\leq \frac g3$.
\end{cor}

One may also compute the cohomology of the Chevalley-Eilenberg complex of $\GCex_{(g),1}$, see Proposition \ref{prop:chevgcex} below.

\section{Notation and preliminaries}
\subsection{Notation}\label{sec:notation}
Unless otherwise stated all vector spaces are taken over the rationals $\Q$. 
We abbreviate the term differential graded by dg. We always use cohomological conventions, that is, differentials have degree $+1$, and we use $\mathbb Z$-gradings unless otherwise stated.
For $V$ a graded vector space we denote by $V[k]$ the same graded vector space with degrees shifted downwards by $k$ units. For example, if $V$ is concentrated in degree $0$, then $V[k]$ is concentrated in degree $-k$.

Almost all objects we consider will be bigraded objects in dg vector spaces or similar categories. That is, these objects come with three gradings, the cohomological grading and and two additional (``weight" and `ìmbalance'') gradings. 
Concretely, we will consider two incarnations of the additional gradings, namely a bigraded dg vector space $V$ may be written either as a direct sum
\[
  V = \bigoplus_{k,l} \gr^{k,l} V 
\]
or as a direct product
\[
  V = \prod_{k,l} \gr^{k,l} V 
\]
of dg sub-vector spaces $\gr^{k,l} V \subset V$.
We will call the second type of grading \emph{complete gradings}.
For example, the dual vector space $V^*$ of a dg vector space with additional grading $V=\bigoplus_{k,l} \gr^{k,l} V$ has a complete grading 
\[
  V^* = \prod_{k,l} \left(\gr^{k,l} V\right)^*.
\]
Often it is hence helpful to consider instead the graded dual dg vector space 
\[
V^c := \bigoplus_{k,l} \left(\gr^{k,l} V\right)^*. 
\]

If $V$ is equipped with further algebraic structure, for example a dg Lie or dg commutative algebra structure, then we will say that the additional gradings are compatible with that algebraic structure if the defining algebraic operations restrict to morphisms of dg vector spaces 
\[
\gr^{k,l} V \times \gr^{k',l'} V \to \gr^{k+k',l+l'} V.
\]

\subsection{Fundamental Theorems of Invariant Theory for the general linear Group}
In our proofs we will use the first and second fundamental theorem of invariant theory for the general linear group. Let $V$ be a finite dimensional vector space over the field $\Q$ and let $\GL(V)$ be the general linear group acting diagonally on the tensor product $V^{\otimes n}$. Additionaly, the left action of the symmetric group $S_n$ on $V^{\otimes n}$ is given by the permutation of the factors
\[
\sigma(v_1,v_2,\dots,v_n) = \left(v_{\sigma^{-1}(1)},v_{\sigma^{-1}(2)},\dots,v_{\sigma^{-1}(n)}\right).
\]
This map can be extended by linearity to a ring homomorphism $\Q[S_n] \rightarrow End\left(V^{\otimes n}\right)^{GL(V)}\cong \left(V^{\otimes n}\otimes (V^*)^{\otimes n}\right)^{GL(V)}$. From the first and second fundamental theorem of invariant theory for the general linear group as stated in theorem 9.1.2 and 9.1.3 in \cite{Loday} we can deduce the following theorem.
\begin{thm}[Fundamental Theorem of Invariant Theory] 
\label{fft}
Let $A,B\geq 0$ be integers, let $V$ be a finite dimensional vector space, and consider the space of invariants 
\[
\left(V^{\otimes A}\otimes (V^*)^{\otimes B}\right)^{\GL(V)}.
\]
If $A\neq B$ then this space is zero. If $A=B$ then the map 
\begin{equation}\label{equ:Sn inv map}
	\Q[S_A] \rightarrow 
	\left(V^{\otimes A}\otimes (V^*)^{\otimes A}\right)^{\GL(V)}.
\end{equation}
is surjective.  If furthermore $dim(V)\geq A$ then \eqref{equ:Sn inv map} is also injective and thus an isomorphism.
\end{thm}

\subsection{Koszul duality}

We refer to \cite{FNW} for a recollection of the Koszul duality theory and the relation to vanishing theorems for cohomology.
Here we shall need only a slight variation of \cite[Proposition 47]{FNW} as follows.

Let us first introduce some notation, cf. \cite[section 2]{FNW}. Let $V$ be a dg vector space. Then the (cohomological-degree-)truncation of $V$ is
\[
 (\tru^{\leq k}V)^j
 =
\begin{cases}
  V^j & \text{for $j<k$} \\
  \{x\in V^k\mid dx=0\} & \text{for $j=k$} \\
  0 & \text{for $j>k$}
\end{cases}.
\]
Here $V^j$ refers to the part of $V$ of cohomological degree $j$.
Next, suppose that $V$ is equipped with an additional $\Z\times \Z$-grading, so that it is totally tri-graded.
For $\alpha,\beta$ numbers we then denote by
\begin{equation}\label{equ:H sq alpha}
  H^{[\alpha,\beta]}V = \bigoplus_{k,l} \gr^{k,l} H^{k\alpha+l\beta}(V)[-k\alpha-l\beta]
\end{equation}
the trigraded vector space whose part of additional degree $(k,l)$ is concentrated in cohomological degree $\alpha k+\beta l$ and agrees with the cohomology of $\gr^{k,l} V$ there.

Furthermore, we set 
\begin{equation}\label{equ:tru sq alpha}
  \tru^{[\alpha,\beta]}V = \bigoplus_{k,l} \gr^{k,l} \tru^{\leq k\alpha+l\beta}(V).
\end{equation}
We note that if $V$ carries further algebraic structure (say a dg Lie or dg commutative algebra structure), compatible with the additional grading, then $\tru^{[\alpha,\beta]}V$ and $H^{[\alpha,\beta]}(V)$ inherit that structure.
Furthermore, one always has a zigzag of dg vector spaces with additional grading 
\[
V \leftarrow  \tru^{[\alpha,\beta]}V \to H^{[\alpha,\beta]}(V)
\]
given by the natural inclusion and projection.
This zigzag also preserves the given algebraic structure on $V$ if present.

\begin{prop}[Essentially Proposition 47 of \cite{FNW}]
	\label{prop:cohom}
	Let $W_0\geq 2$ be an integer and $\alpha,\beta\in \Q$.
	Suppose that $\fg$ is a dg Lie algebra with an additional $\Z_{>0}\times \Z$-grading.
	We refer to the two grading parameters as the ``weight" $W\geq 1$ and imbalance $H\in \Z$.
	Suppose that $\fg$ satisfies the following properties.
	\begin{itemize}
	  \item Each graded piece $\gr^{(W,H)}\fg$ is finite dimensional.
	  \item $\gr^{(W,H)} H(\fg)$ is concentrated in cohomological degree $\alpha W + \beta H$ for any $W\leq W_0$ and $H\in \Z$.
	  \item The cohomology of the weight $(W,H)$ piece of the cobar construction $\gr^{(W,H)} \Bar^c \fg^c$ is concentrated in cohomological degree $(1-\alpha)W-\beta H$ for any $W\leq W_0$ and $H\in \Z$.
	\end{itemize}
	Then the following hold:
	\begin{enumerate}[(i)]
	  \item $\fg$ is formal up to weight $W_0$ in the sense that the zigzag of morphisms of dg Lie algebras
	  \[
	  \fg \leftarrow \tru^{[\alpha,\beta]}\fg \to H^{[\alpha,\beta]}(\fg)   
	  \] 
	(see \eqref{equ:H sq alpha}, \eqref{equ:tru sq alpha} for the notation) induces a zigzag of quasi-isomorphisms of complexes on the part of weight $W$ for any $W\leq W_0$.
	  \item Similarly, the dg commutative algebra $\Bar^c \fg^c$ is formal up to weight $W_0$ in the sense that the zigzag of morphisms of dg commutative algebras 
	  \[
		\Bar^c \fg^c \leftarrow \tru^{[1-\alpha,-\beta]}
		\Bar^c \fg^c \to H^{[1-\alpha,-\beta]}(\Bar^c \fg^c)
	  \]
	  induces quasi-isomorphisms on the part of weight $W\leq W_0$.
  
	  \item Let $V:=\gr^1 H(\fg)\cong \gr^1 H(\Bar \fg)[-1]$, $R:= \gr^2H(\Bar \fg)$ and $S:=\gr^2H(\fg^c)$. Then the maps 
	  \begin{align}\label{equ:kos2 inj}
		R &\to \wedge^2V & S\to S^2(V[-1]^*)  
	  \end{align}
	  given by the (reduced) coproduct and the cobracket are injections. We denote the images by $R$ and $S$ as well, abusing notation.
	  Furthermore, $S$ is (up to degree shift) the annihilator of $R$ if we identify $(\wedge^2V)^*\cong S^2(V[-1]^*)[2]$.
	  \item The graded Lie algebra $\ft$ and graded commutative algebra $A$ defined via the quadratic presentations
	  \begin{align*}
		\ft &= \FreeLie(V)/\langle R\rangle & 
		A= S(V^*[-1]) / \langle S\rangle.
	  \end{align*}
	  are quadratic duals of each other.
	  \item The maps of graded Lie, respectively graded commutative algebras
	  \begin{align}\label{equ:xxx3}
		\ft &\to H(\fg) & 
		A&\to H(\Bar^c \fg^c)
	  \end{align}
	  defined by the obvious inclusion of generators are well-defined, respect the weight and imbalance gradings and induce isomorphisms on the parts of weight $W$ and imbalance $H$ for each $W\leq W_0$ and $H\in \Z$.
	  \item The pair $\ft$ and $A$ are Koszul up to weight $W_0$ in the sense that the cohomology of $\gr^{(W,H)}\Bar \ft$ (respectively $\gr^{(W,H)}\Bar A$) is concentrated in cohomological degree $(\alpha-1) W+\beta H$ (respectively $-\alpha W-\beta H$) for $W\leq W_0$, $H\in \Z$.

	  Equivalently, the maps 
	  \begin{align*}
		\Bar^c(A^c) &\to \ft 
		&
		\Bar^c(\ft^c)&\to A 
	  \end{align*}
	  are quasi-isomorphisms in weight $\leq W_0$.
	\end{enumerate}
  \end{prop}
  \begin{proof}
	The version of this proposition without imbalance grading is shown in \cite{FNW}. Our variant follows by repeating the proof, and restricting to pieces of fixed imbalance $H$ where appropriate.
  \end{proof}

%
%
%
%

\subsection{Extension of the general linear Lie algebra, $\osp_{g}^{nil}$ and $\osp_{g,1}^{nil}$} 
\label{ospnil}
We extend the general linear Lie algebra $\mathfrak{gl_{g}}$ to $\osp_g:=\mathfrak{gl_{g}}\ltimes\osp_g^{nil}$, where $\osp_g^{nil}$ denotes the nilpotent part composed of endomorphisms $\phi$ of $H^\bullet(W_{g})$ with degree $deg(\phi)<0$ which leave the antisymmetric Poincaré duality pairing
\[
\langle -,-\rangle: H^\bullet(W_{g})\times H^\bullet(W_{g}) \to \Q
\]
invariant. Let $\{1,\alpha_1,\dots,\alpha_g,\beta_1,\dots,\beta_g,\omega\}$ denote a basis of $H^\bullet(W_{g})$ such that the pairing is $\langle \alpha_i,\beta_j\rangle=\delta_{ij}=-\langle \beta_j,\alpha_i\rangle$ and $\langle 1,\omega\rangle=1=-\langle \omega,1\rangle$ with all other pairings between basis elements being zero. We have $H^0(W_g)=\Q1$, $H^m(W_g)=span\{\alpha_1,\dots,\alpha_g\}$, $H^{m+1}(W_g)=span\{\beta_1,\dots,\beta_g\}$, and $H^{2m+1}(W_g)=\Q\omega$.

\begin{lemma}
	$\osp_{g}^{nil}$ consists of maps of three different degrees, which leave the Poincaré duality paring invariant:
	\begin{enumerate}[i)]
		\item $deg(\phi) = -1:\beta_i \to \sum_j A_{ij}\alpha_j$ with $A=(-1)^{m+1}A^T$
		\item $deg(\phi) = -m:\omega \to \sum_j \lambda_j\beta_j$, $\alpha_i\to(-1)^{m+1}\lambda_i1$
		\item $deg(\phi) = -m-1:\omega \to \sum_j \mu_j\alpha_j$, $\beta_i\to(-1)^{m+1}\mu_i1$	
	\end{enumerate}
\end{lemma}

\begin{proof}\
\begin{enumerate}[i)]
	\item 
	\begin{align*}
		\left\langle\phi(\beta_i),\beta_j\right\rangle + (-1)^{|\phi||\beta_i|}\left\langle\beta_i,\phi(\beta_j)\right\rangle 
		&= \sum_{k}A_{ik}\langle\alpha_k,\beta_j\rangle + (-1)^{m+1}\sum_{l}A_{jl}\langle\beta_i,\alpha_l\rangle\\
		&=	A_{ij}-(-1)^{m+1}A_{ji}=0
	\end{align*}
	$\implies A=(-1)^{m+1}A^T$
	\item
	$\omega \to \sum_j \lambda_j\beta_j$, $\alpha_i\to\tilde{\lambda}_i1$
	\begin{align*} 
		\left\langle\phi(\omega),\alpha_j\right\rangle + (-1)^{|\phi||\omega|}\left\langle\omega,\phi(\alpha_j)\right\rangle 
		&= \sum_{i}\lambda_i\langle\beta_i,\alpha_j\rangle + (-1)^{m}\tilde{\lambda}_j\langle\omega,1\rangle\\
		= -\left(\lambda_j+(-1)^m\tilde{\lambda}_j\right)=0
	\end{align*}
	$\implies \tilde{\lambda}=(-1)^{m+1}\lambda$
	\item
	$\omega \to \sum_j \mu_j\alpha_j$, $\beta_i\to\tilde{\mu}_i1$
	\begin{align*} 
		\left\langle\phi(\omega),\beta_j\right\rangle + (-1)^{|\phi||\omega|}\left\langle\omega,\phi(\beta_j)\right\rangle 
		=& \sum_{i}\mu_i\langle\alpha_i,\beta_j\rangle + (-1)^{m+1}\tilde{\mu}_j\langle\omega,1\rangle\\
		=& \mu_j-(-1)^{m+1}\tilde{\mu}_j=0
	\end{align*}
	$\implies \tilde{\mu}=(-1)^{m+1}\mu$
\end{enumerate}
\end{proof}

Let $\tilde\Delta \in S^2(H^\bullet(W_{g}))$ be the canonical diagonal element:

\[
\tilde\Delta = 1\otimes \omega -\omega\otimes 1+(-1)^m \sum_{i=1}^g (\alpha_i\otimes \beta_i - \beta_i\otimes \alpha_i).
\]

\begin{lemma}
	Let $x\in osp_g^{nil}$ then $x\cdot\tilde\Delta=0$.
\end{lemma}

\begin{proof}\
\begin{enumerate}[i)]
	\item 
	\[
		x\cdot\tilde\Delta = (-1)^m\sum_{i,j} A_{ij}\left(\beta_j\otimes \beta_i-(-1)^{|x||\beta_i|}\beta_i\otimes\beta_j\right)
		= (-1)^m\left((-1)^{(m+1)(m+1)}-(-1)^{m(m+1)}\right)\sum_{i,j} A_{ij}\beta_i\otimes\beta_j
		=0
	\]
	with $(-1)^{(m+1)(m+1)}-(-1)^{m+1} = (-1)^{m^2+2m+1}-(-1)^{m+1}= (-1)^{m+1}-(-1)^{m+1}=0$
	\item
	\begin{align*}
		x\cdot\tilde\Delta &= \sum_i\lambda_i\left(1\otimes\beta_i-\beta_i\otimes1\right)+(-1)^m\sum_i(-1)^{m+1}\lambda_i\left(1\otimes\beta_i-(-1)^{|x||\beta_i|}\beta_i\otimes 1\right)\\
		&= \sum_i\lambda_i\left(1\otimes\beta_i-\beta_i\otimes1-\left(1\otimes\beta_i-(-1)^{m(m+1)}\beta_i\otimes 1\right)\right)\\
		&= \sum_i\lambda_i\left(1\otimes\beta_i-\beta_i\otimes1-\left(1\otimes\beta_i-\beta_i\otimes 1\right)\right)\\
		&=0
	\end{align*}
	\item
	\begin{align*}
		x\cdot\tilde\Delta &= \sum_i\mu_i\left(1\otimes\alpha_i-\alpha_i\otimes1\right)+(-1)^m\sum_i(-1)^{m+1}\mu_i\left(1\otimes\alpha_i-(-1)^{|x||\alpha_i|}\alpha_i\otimes 1\right)\\
		&= \sum_i\mu_i\left(1\otimes\alpha_i-\alpha_i\otimes1-\left(1\otimes\alpha_i-(-1)^{(m+1)m}\alpha_i\otimes 1\right)\right)\\
		&= \sum_i\mu_i\left(1\otimes\alpha_i-\alpha_i\otimes1-\left(1\otimes\alpha_i-\alpha_i\otimes 1\right)\right)\\
		&=0
	\end{align*}
\end{enumerate}	
\end{proof}

$\osp_{g,1}:=\mathfrak{gl}_g\ltimes\osp_{g,1}^{nil}$ is defined in the same way, but without the top class $\omega$. In particular $\osp_{g,1}^{nil}$ is composed just of the endomorphisms of degree $-1$
\[
	\osp_{g,1}^{nil} := \{ x\in \osp_{g}^{nil} \mid |x|=-1 \}.
\]

\subsection{Preliminary Lemmas from Homological Algebra} 

We shall use the following three standard results of homological algebra.

\begin{lemma}
\label{lemma:inclusion}
Suppose $(V,d)$ is a dg vector space and $U\subset V$ is a dg subspace, i.e., $dU\subset U$.
Then the inclusion $U\hookrightarrow V$ is a quasi-isomorphism if and only if the quotient dg vector space $V/U$ is acyclic, i.e. $H(V/U)=0$.
\end{lemma}
\begin{proof}
	Let $cone(i)$ be the mapping cone of the inclusion $i:U\hookrightarrow V$.
	Then we have:

$i:U\hookrightarrow V$ is a quasi-isomorphism\\
$\Leftrightarrow H(cone(i))=0$\\
$\Leftrightarrow \forall (u,v) \in U\oplus V$ with $(-du,dv-i(u))=0$, $\exists(\tilde u,\tilde v)\in U\oplus V$ such that $(-d\tilde u, d\tilde v -i(\tilde u)) =(u,v)\quad (1)$\\

On the other hand:
$H(V/U)=0$\\
$\Leftrightarrow\forall v\in V$ with $dv \in U$, $\exists \tilde v \in V$ with $d\tilde v-v\in U$\\
$\Leftrightarrow \forall v\in V$ with $dv=i(u)$ for some $u\in U$, $\exists \tilde v \in V, \tilde u \in U$ with $d\tilde v-v= i(\tilde u)\quad (2)$\\

$(1)\Rightarrow(2)$:\\
The precondition of $(2)$ implies the precondition of $(1)$ since $du=i(du)=di(u)=d^2v=0$ wehre we used that $U$ is a dg subspace. The postcondition of $(1)$ implies the postcondition of $(2)$.

$(2)\Rightarrow(1)$:\\
The precondition of $(1)$ implies the precondition of $(2)$. The postcondition of $(2)$ implies the postcondition of $(1)$ since $d\tilde u=i(d\tilde u)=di(\tilde u)=d^2\tilde v-dv=i(u)=u$ where we used that $U$ is a dg subspace.
\end{proof}


\begin{lemma}
\label{lemma:graded}
Let $(V,d)$ be a dg vector space equipped with a bounded filtration $\mF^\bullet(-)$.
If the associated graded dg vector space $\gr V$ is acyclic, so is $V$.
\end{lemma}
\begin{proof}
	Without loss of generality we assume that the filtration is descending and therefore $gr^pV^k=\mF^pV^k/\mF^{p+1}V^k$. Let $\alpha\in V^k$ with $d\alpha=0$, we show that $\exists \beta \in V^{k-1}$ with $\alpha=d\beta$.
	 
	Since the filtration is bounded $\exists p > 0$ such that $\alpha \in \mF^pV^k$. Due to the assumption $\exists \beta_p \in \mF^pV^{k-1}$ such that $[d\beta_p]=[\alpha]$ in $gr^pV^k$. We define $\alpha_{p+1}:=\alpha-d\beta_p$ with $\alpha_{p+1}\in \mF^{p+1}V^k$ and $d\alpha_{p+1}=0$.
	
	If we repeat this step inductively we reach a $N>0$ such that $\alpha_N:= \alpha_{N-1}-d\beta_{N-1}\in \mF^{N}V^k=0$ because the filtration is bounded. Hence, $\alpha = \sum_{i=p}^Nd\beta_i=d\sum_{i=p}^N\beta_i$.
	\end{proof}

\begin{lemma}
\label{lem:homotopy}
	Let $(V,d)$ be a dg vector space and let $h:V\to V$ be a linear map of degree $-1$ such that $d h+h d = \mathds{1} $ is the identity.
	Then $(V,d)$ is acyclic. 
	\hfill\qed
\end{lemma}

\subsection{Representations of $\GL_g$}
Let $V_g$ denote the defining representation of $\GL_g$ and $V_g^*$ the dual representation. We say that a $\GL_g$-representation is of order $r$ if it is a direct sum of sub-quotients of representations of the form  $V_g^k \otimes \left(V_g^*\right)^l$ with $k+l\leq r$. 

\begin{lemma}
\label{lem:rep_gl_g}
Let V be a $\GL_g$-representation of order $r$. If 
\[
\left(V\otimes V_g^{\otimes k}\otimes \left(V_g^*\right)^{\otimes l}\right)^{GL_g}=0 \quad\forall k,l \text{ such that } k+l\leq r,
\]
then we have $V=0$.
\end{lemma}
\begin{proof}
We may assume that $V\neq 0$ is a subquotient of some $V_g^{\otimes k}\otimes \left(V_g^*\right)^{\otimes l}$, i.e., $V=U/W$ with
\[
	W\subsetneq U \subset  V_g^{\otimes k}\otimes \left(V_g^*\right)^{\otimes l} 
\]
some $\GL_g$-invariant subspaces. Then it suffices to check that 
\begin{equation}\label{equ:lem rep glg proof}
\left(V\otimes V_g^{\otimes l}\otimes \left(V_g^*\right)^{\otimes k}\right)^{GL_g}
= 
\Hom_{\GL_g}\left(V_g^{\otimes k}\otimes \left(V_g^*\right)^{\otimes l}, V\right)\neq 0.
\end{equation} 
Use the complete reducibility of $V_g^{\otimes k}\otimes \left(V_g^*\right)^{\otimes l}$ to pick a $\GL_g$-invariant complement $U'$ of $U$. For example $U'=U^\perp$ the orthogonal complement with respect to the standard inner product will do.
Then the composition of the natural projections 
\[
V_g^{\otimes k}\otimes \left(V_g^*\right)^{\otimes l}
= U\oplus U' \to U \to U/W
\]
is surjective and hence nonzero, and therefore is a non-zero element in \eqref{equ:lem rep glg proof}.
\end{proof}

\section{Graph complexes}\label{sec:GCs}

In this section we give a combinatorial definition of the graph complexes $\G_{(g),1}$, $\GC_{(g),1}:=\G_{(g),1}^*$ and $ \fG_{(g),1}\cong S^+(\G_{(g),1}[-1])$ considered in this paper. 
These complexes or close variants therof have appeared at other places in the literature \cite{CamposWillwacher,Matteo}.

We say that a directed graph with $n$ vertices and $k$ edges is an ordered set of $k$ pairs $(i,j)$ of numbers $i,j\in \{1,\dots,n\}$. These $k$ pairs are the edges of the graph, with the edge $(i,j)$ pointing from vertex $i$ to vertex $j$. We do not ask that the graph is connected.
We denote the set of such graphs with $n$ vertices and $k$ edges by $\gra_{n,k}$. This set carries an action of the group $S_n\times (S_2 \wr S_k)$, with $S_n$ acting by renumbering the vertices, $S_k$ by reordering the edges, and $S_2$ by changing the directions of the edges, i.e., by flipping the two members of the pairs. 

Let furthermore $V$ be any finite dimensional graded vector space and $m$ an integer.
We then define a graded vector space of coinvariants
\begin{equation}\label{equ:fGdef}
  \fG_{V,2m+1} := \bigoplus_{n\geq 1, k\geq 0}
  \left( \Q\gra_{n,k}[2mk]\otimes \Q[2m+1]^{\otimes n} \otimes sgn_{S_2} ^{\otimes k} 
  \otimes (S(V))^{\otimes n} \right)_{S_n\times (S_2 \wr S_k)}. 
\end{equation}
Here the group $S_n$ acts diagonally on the vector space $\Q\gra_{n,k}$ generated by the set $\gra_{n,k}$ and the $n$ symmetric product factors $S(V)$ as well as by permuting the factors $\Q[2m+1]$, thus producing Koszul signs. $S_k$ acts on $\Q\gra_{n,k}$.
We may interpret elements of $\fG_{V,2m+1}$ as linear combinations of isomorphism classes of graphs, whose vertices may be decorated by zero or more elements of $V$. Additionally, such graphs come with an ordering of the vertices and the decorations as well as with an orientation of the edges. We identify two such orderings and orientations up to sign. 

We say that the valence of a vertex in a graph is the number of incident half-edges, plus the number of decorations.
E.g., a vertex $v$ with $3$ incident half-edges and a decoration in $S^2(V)$ has valence $5$.
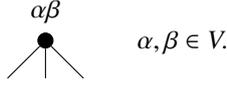
\[
\begin{tikzpicture}
\node[int, label=90:{$\alpha\beta$}] (v) at (0,0) {};
\draw (v) edge +(-.5, -.5) edge +(0,-.5) edge +(.5,-.5);
\end{tikzpicture}
\quad \quad \alpha, \beta\in V.
\]

In the following we restrict ourselves to the subcomplex of $\fG_{V,2m+1}$ spanned by graphs all of whose vertices are at least trivalent and use the notation $\fG_{V,2m+1}^{\geq 3}$. 
We define on $\fG_{V,2m+1}^{\geq 3}$ a differential $d_c$ such that for a graph $\Gamma\in \fG_{V,2m+1}$
\[
  d_c \Gamma = \sum_{e=(i,j)\in E\Gamma} (-1)^j \Gamma / e,
\]
with the sum running over all edges $e$ in the edge set $E\Gamma$, and obtain the graph $\Gamma/e$ by contracting the edge $e$.
The operation of contracting some edge $e=(i,j)$ is precisely the following.
If $e$ is a tadpole edge, i.e., $i=j$ then the graph is zero by symmetry and we set $\Gamma/e=0$. 
Otherwise we remove the edge $e$ from the list of edges and the vertex $j$ from the set of vertices, replace all occurrences of $j$ by $i$ in the list of edges and renumber the vertices again by $1,2,\dots,n-1$, keeping their relative order. Finally we alter the decorations by applying the natural map 
  \[
  S(V)^{\otimes n}\to  S(V)^{\otimes n-1} 
  \] 
  by multiplying the $j$-th factor into the $i$-th factor, using the commutative product on the symmetric algebra $S(V)$. This operation implicitly includes Koszul signs for permuting the $j$-th factor to the $i$-th position.



Next we restrict to the case considered in this paper where $V=\bar H^\bullet(W_{g,1})$. We will denote the corresponding graded vector space by 
\[
	\fG_{(g),1} := \fG^{\geq 3}_{\bar H^\bullet(W_{g,1}),2m+1}.
\]
Note that we have a canonical "diagonal" element $\Delta \in S^2(\bar{H}^\bullet(W_{g,1}))$ of degree $2m+1$.
Let $\{\alpha_1,\dots, \alpha_g,\beta_1,\dots,\beta_g\}$ denote a basis of $\bar{H}^\bullet(W_{g,1})$ such that the Poincar\'e duality pairing is $\langle \alpha_i,\beta_j\rangle = \delta_{ij}$, with all other pairings zero. 
Then the canonical diagonal element is given by
\[
\Delta = (-1)^m \sum_{i=1}^g (\alpha_i\otimes \beta_i - \beta_i\otimes \alpha_i).
\]

We then define a second differential $d_{cut}$ on a graph $\Gamma$ with $|V\Gamma|$ vertices such that 
\[
d_{cut} \Gamma = (-1)^{|V\Gamma|} \sum_{e\in E\Gamma} \cut(\Gamma,e)   
\]
where the sum is again over edges and $\cut(\Gamma,e)$ is defined as follows:
If we cut the edge $e=(i,j)$, then $\cut(\Gamma,e)$ is obtained by removing $e$, and multiplying $\Delta$ from the left into $S(V)^{\otimes n}$, at positions $i$ and $j$ of the tensor product. There is again an implicit Koszul sign corresponding to the permutation of the two factors of $\Delta\in S(V)^{\otimes 2}$ to the $t$-th and $j$-th position.

Finally, the total differential on the considered graph complex $\fG_{(g),1}$ with all vertices at least trivalent reads as
\[
d = d_c + d_{cut}.
\]
\begin{lemma}
The complex $(\fG_{(g),1},d)$ with all vertices at least trivalent is well-defined, in particular $d^2=0$. 
\end{lemma}
\begin{proof}
	We need to check separately that 
	\begin{align*}
		d_c^2&=0 &
		d_{cut}^2&=0 &
		d_cd_{cut}+d_{cut}d_c&=0.
	\end{align*}
First, we compute 
	\begin{align*}
		d_c^2\Gamma 
		=
		\sum_{f=(f_1,f_2)\in E\Gamma \atop 
		f\neq e}
		\sum_{e=(e_1,e_2)\in E\Gamma}
		(-1)^{e_2 + \tilde f_2}
		(\Gamma / e) / \tilde f
	\end{align*}
where we denote by $\tilde f=(\tilde f_1,\tilde f_2)$ the edge in the graph $\Gamma / e$ that corresponds to $f\in E\Gamma$.
Similarly, let us denote by $\tilde e=(\tilde e_1,\tilde e_2)$ the edge in $\Gamma / f$ corresponding to $e\in E\Gamma$.
Then the fact that $d_c^2=0$ follows immediately from the following claim by antisymmetry under the interchange of $e$ and $f$.

\medskip 

{\bf Claim:} We have 
\[
(-1)^{e_2 + \tilde f_2}
(\Gamma / e) / \tilde f
	=
-(-1)^{f_2 + \tilde e_2}
	(\Gamma / f) / \tilde e.
\]
To show the claim, we distinguish two cases:

\begin{itemize}
\item $e_2\neq f_2$: 
Without loss of generality we can assume $e_2>f_2$, otherwise interchange the roles of $e$ and $f$.
In this case $\tilde e_2= e_2-1$ and $\tilde f_2=f_2$.
Hence $(-1)^{e_2 + \tilde f_2} = -(-1)^{f_2 + \tilde e_2}$.
Furthermore, $(\Gamma / e) / \tilde f= (\Gamma / f) / \tilde e$, so that the claim holds.
\item $e_2= f_2$:
Both edges have the same endpoint $e_2$:
\[
\begin{tikzpicture}	
	\node[int,label=90:{$e_1$}] (e1) at (0,0){};
\node[int,label=-90:{$e_2=f_2$}] (e2) at (1.2,0){};
\node[int,label=90:{$f_1$}] (f1) at (2.4,0){};
\draw[-latex] (e1) edge node[above]{$e$} (e2) (f1) edge node[above]{$f$} (e2);
\end{tikzpicture}
\]

We may assume that the starting points of the edges are different, $e_1\neq f_2$, since otherwise $(\Gamma / e) / \tilde f=(\Gamma / f) / \tilde e=0$.
We furthermore have $(\tilde e_1,\tilde e_2)=(\tilde f_2,\tilde f_1)$. 
Also note that contracting edges $e$ and $f$ in either order produces isomorphic graphs. However, the new vertex (the image of $e_1,e_2$ and $f_2$) inherits its position in the ordering of vertices either from $e_1$ or $f_1$. Hence 
\[
	(\Gamma / e) / \tilde f = -(-1)^{\tilde e_2 -\tilde e_1}	(\Gamma / f) / \tilde e
\]
and thus
\[
	(-1)^{e_2+\tilde f_2}(\Gamma / e) / \tilde f
	=
	(-1)^{f_2+\tilde e_1}
	(\Gamma / e) / \tilde f
	=
	-(-1)^{f_2 +\tilde e_1 + \tilde e_2-\tilde e_1}	(\Gamma / f) / \tilde e
	= -(-1)^{f_2 + \tilde e_2}
	(\Gamma / f) / \tilde e
\]
as desired, showing the claim.
\end{itemize}

Next note that for $e\neq f\in E\Gamma$:
\[
	cut(cut(\Gamma,e),f) =
	-cut(cut(\Gamma,f),e),
\]
with the sign produced by the Koszul sign rule for multiplying the to copies of $\Delta$ into the appropriates positions in $S(V)^{\otimes n}$.
Hence we trivially have 
\[
d_{cut}^2 \Gamma 
=
\sum_{e\neq f\in E\Gamma}cut(cut(\Gamma,e),f)
=0
\]
by symmetry under interchange of $e$ and $f$.

Finally, we have, again with the same notation as above,
\[
cut(\Gamma/e,\tilde f) =
cut(\Gamma,f)/e.
\]
Hence 
\begin{align*}
	d_{cut}d_c \Gamma &= 
\sum_{e=(i,j)\in E\Gamma}
\sum_{f\in E\Gamma\atop f\neq e}
(-1)^{|V\Gamma/e|}(-1)^j 
cut(\Gamma/e,\tilde f)
\\&=
-
\sum_{f\in E\Gamma}
\sum_{e=(i,j)\in E\Gamma \atop e\neq f}
(-1)^{|V\Gamma|}(-1)^j 
cut(\Gamma,f)/e
\\&=  -d_{c}d_{cut} \Gamma
\end{align*}

\end{proof}

Suppose that $\Gamma$ is a graph in the graph complex $\fG_{(g),1}$, with $e$ edges, $v$ vertices, $a$ decorations in $\bar H^m(W_{g,1})$ and $b$ decorations in $\bar H^{m+1}(W_{g,1})$. Then the cohomological degree of $\Gamma$ is 
\[
k=2m e - (2m+1) v +ma + (m+1)b.  
\]
With $D:=a+b$ the total number of decorations we furthermore define the \emph{weight} of $\Gamma$ as 
\[
W=2(e-v) +D.
\]
The imbalance is defined as 
$$
H=a-b.
$$ 
The $\Z\times \Z$-grading with respect to the weight $W$ and imbalance $H$ is preserved by the differential. Hence, the graph complex splits into a direct sum of pieces of fixed weight and imbalance.

Next, note that the graph complex $\fG_{(g),1}$ is equipped with a commutative product by taking disjoint union of graphs. This product is compatible with the differential and the weight and imbalance gradings, since the weight and imbalance are additive.
Furthermore, any graph splits uniquely (up to reordering) as a union of connected graphs. It follows that our graph complex is quasi-free as a (non-unital) graded commutative algebra, so that we can write 
\[
  \fG_{(g),1}\cong S^+(\G_{(g),1}[-1]).
\]
Here we denoted by $\G_{(g),1}[-1]\subset \fG_{(g),1}$ the subcomplex spanned by connected graphs.
Since the differential is compatible with the commutative algebra structure (i.e., a derivation), it endows the degree shifted connected part $\G_{(g),1}$ with a dg Lie coalgebra structure.
We shall denote the dual dg Lie algebras by 
\[
  \GC_{(g),1}:=\G_{(g),1}^* 
\]
and the dual differential by
\[
\delta = d^* = d_c^* + d_{cut}^* = \delta_{split} + \delta_{glue}.
\]
The combinatorial expressions for the different pieces of the (dual) differential are given in the introduction.

\section{Generators and Relations of the Cohomology in terms of irreducible Representations of $\GL_g$}
\label{sec:representations}

In this section we express the generators and relations of the non vanishing cohomology in therms of irreducible representations of $GL_g$ for large genera. Let $V_{g}=\Q^g$ the $g$-dimensional defining representation of $\GL_g$, then we have that $H^m(W_{g,1})\cong V_g$ and $H^{m+1}(W_{g,1})\cong V_g^*$, as $\GL_g$-representations. Furthermore let $V_{(\lambda_1,\lambda_2,\dots,\lambda_g)}$ denote the irreducible representation of $\GL_g$ corresponding to the non-increasing sequence of integers $\lambda_1\geq\lambda_2\geq\dots\geq\lambda_g$ with $\lambda_i\in\Z$, see \cite[\S 15.5]{FultonHarris}. In this notation we have $V_g=V_{(1,0,\dots)}$ and $V_g^*=V_{(0,\dots,0,-1)}$. In the following
\[
H = \text{\#decorations in $H^m\left(W_{g,1}\right)$} - \text{\#decorations in $H^{m+1}\left(W_{g,1}\right)$}
\]
denotes again the imbalance. Furthermore, we assume that $g\geq 6$ for simplicity. 

\subsection{Decomposition of $\gr^1H\left(\GC_{(g),1}\right)$}
First we consider $\gr^1H\left(\GC_{(g),1}\right)$ which corresponds to graphs composed of a single internal vertex with three decorations in $V_g \oplus V_g^*$
\[
  	\begin{tikzpicture}
  		\node[int] (v1) at (0,0) {};
  		\node at (0:0.4) {$\alpha$};
  		\node at (120:0.4) {$\beta$};
  		\node at (-120:0.4) {$\gamma$};
  	\end{tikzpicture}
  	\quad \quad \text{with} \quad \alpha,\beta,\gamma \in V_g \oplus V_g^*.
\]

The following table shows the decomposition of $\gr^1H\left(\GC_{(g),1}\right)$ in terms of irreducible representations of $GL_g$.
\begin{align}
\begin{tabular}{|l|l|l|}
 \hline
 H & m even & m odd \\
 \hline
 3 & $S^3(V)=V_{(3,0,\dots,0)}$ & $\wedge^3(V)=V_{(1,1,1,0,\dots,0)}$\\
 \hline
 1 & $S^2(V)\otimes V^*=V\oplus V_{(2,0,\dots,0,-1)}$ & $\wedge^2(V)\otimes V^*=V\oplus V_{(1,1,0,\dots,0,-1)}$\\
 \hline
 -1 & $V\oplus\wedge^2(V^*)=V^*\oplus V_{(1,0,\dots,0,-1,-1)}$ & $V\otimes S^2(V^*)=V^*\oplus V_{(1,0,\dots,0,-2)}$ \\
 \hline
 -3 & $\wedge^3(V^*)=V_{(0,\dots,0,-1,-1,-1)}$ & $S^3(V^*)=V_{(0,\dots,0,-3)}$\\
 \hline
\end{tabular}
\label{table:generators}
\end{align}

\subsection{Decomposition of $\gr^2H\left(\GC_{(g),1}\right)$}
In a second step we consider $\gr^2\left(\GC_{(g),1}\right)$. There are 4 possible types of graphs contributing to weight 2, with at most 2 vertices:
\begin{align*}
	A&= \begin{tikzpicture}
		\node[int] (v1) at (0,0) {};
	  \node[int] (v2) at (1,0) {};
	  \draw (v1) edge (v2) edge[bend left] (v2) edge[bend right] (v2);
	\end{tikzpicture}
	&
	B&=
	\begin{tikzpicture}
		\node[int,label=180:{$\alpha$}] (v1) at (0,0) {};
	  \node[int,label=0:{$\beta$}] (v2) at (1,0) {};
	  \draw (v1) edge[bend left] (v2) edge[bend right] (v2);
	\end{tikzpicture}
	& C&=
  	\begin{tikzpicture}
  		\node[int] (v1) at (0,0) {};
  		\node at (45:0.4) {$\alpha$};
  		\node at (135:0.4) {$\beta$};
  		\node at (-135:0.4) {$\gamma$};
  		\node at (-45:0.4) {$\delta$};
  	\end{tikzpicture}
  	&D&=
	\begin{tikzpicture}
  		\node[int] (v1) at (0,0) {};
  		\node at (0,0.4) {$\alpha$};
  		\node at (0,-0.4) {$\beta$};
		\node[int] (v2) at (1,0) {};
		\node at (1,0.4) {$\gamma$};
  		\node at (1,-0.4) {$\delta$};
		\draw (v1) edge (v2);
  	\end{tikzpicture}
  	\quad \quad \text{with }\quad \alpha,\beta,\gamma,\delta \in V_g \oplus V_g^*.
\end{align*}


We decompose the corresponding subspaces of $\gr^2\GC_{(g),1}$ into irreducible representations of $GL_g$.

First, graph $A$ above clearly contributes one copy of the trivial representation 
\begin{equation}\label{equ:triv rep}
	V_{0,\dots,0} \quad\quad \text{in imbalance $H=0$}.
\end{equation}

The graphs of type $B$ contribute as follows:
\begin{align}
	\begin{tabular}{|l|l|l|}
	 \hline
	 H & m even & m odd \\
	 \hline
	 2 & $\wedge^2(V)=V_{(1,1,0,\dots,0)}$ & $S^2(V)=V_{(2,0,\dots,0)}$\\
	 \hline
	 0 & $V\otimes V^*=V_{(1,0,\dots,0,-1)}\oplus V_{(0,\dots,0)}$ & $V\otimes V^*=V_{(1,0,\dots,0,-1)}\oplus V_{(0,\dots,0)}$\\
	 \hline
	 -2 & $S^2(V^*)=V_{(0,\dots,0,-2)}$ & $\wedge^2(V^*)=V_{(0,\dots,0,-1,-1)}$ \\
	 \hline
	\end{tabular}
	\label{table:2vertex2edge}
\end{align}

The following table shows the contributions from graphs of type $C$: 

\begin{align}
\begin{tabular}{|l|l|l|}
 \hline
 H & m even & m odd \\
 \hline
 4 & $S^4(V)=V_{(4,0,\dots,0)}$ & $\wedge^4(V)=V_{(1,1,1,1,0,\dots,0)}$\\
 \hline
 2 & $S^3(V)\otimes V^*=V_{(2,0,\dots,0)}\oplus V_{(3,0,\dots,0,-1)}$ & $\wedge^3(V)\otimes V^*=V_{(1,1,0,\dots,0)}\oplus V_{(1,1,1,0,\dots,0,-1)}$\\
 \hline
 0 & $S^2(V)\otimes\wedge^2(V^*)=V_{(1,0,\dots,0,-1)}\oplus V_{(2,0,\dots,0,-1,-1)}$ & $\wedge^2(V)\otimes S^2(V^*)=V_{(1,0,\dots,0,-1)}\oplus V_{(1,1,0,\dots,0,-2)}$\\
 \hline
 -2 & $V\oplus\wedge^3(V^*)=V_{(0,\dots,0,-1,-1)}\oplus V_{(1,0,\dots,0,-1,-1,-1)}$ & $V\otimes S^3(V^*)=V_{(0,\dots,0,-2)}\oplus V_{(1,0,\dots,0,-3)}$ \\
 \hline
 -4 & $\wedge^4(V^*)=V_{(0,\dots,0,-1,-1,-1,-1)}$ & $S^4(V^*)=V_{(0,\dots,0,-4)}$\\
 \hline
\end{tabular}
\label{table:1vertex}
\end{align}

Finally, we have the contributions from graphs of type $D$: 

\begin{align}
\begin{tabular}{|l|l|l|}
 \hline
 H & m even & m odd \\
 \hline
 4 & $S^2(S^2(V))=V_{(2,2,0,\dots,0)}\oplus V_{(4,0,\dots,0)}$ 
 & $S^2(\wedge^2(V))=V_{(1,1,1,1,0,\dots,0)}\oplus V_{(2,2,0,\dots,0)}$\\
 \hline
 2 & 
 $\begin{array}{l}
 S^2(V)\otimes(V\otimes V^*)\\
 =V_{(1,1,0,\dots,0)}\oplus 2V_{(2,0,\dots,0)}\\
 \quad \oplus V_{(2,1,0,\dots,0,-1)}\oplus V_{(3,0,\dots,0,-1)}
 \end{array}$
 & 
 $\begin{array}{l}
 \wedge^2(V)\otimes(V\otimes V^*)\\
 =2V_{(1,1,0,\dots,0)}\oplus V_{(1,1,1,0,\dots,0,-1)}\\
 \quad \oplus V_{(2,0,\dots,0)}\oplus V_{(2,1,0,\dots,0,-1)} 
 \end{array}$\\
 \hline
 0 & 
 $\begin{array}{l}
 S^2(V)\otimes\wedge^2(V^*)\oplus \wedge^2(V\otimes V^*)\\
 =V_{(1,0,\dots,0,-1)}\oplus V_{(2,0,\dots,0,-1,-1)} \\
 \quad \oplus 2V_{(1,0,\dots,0,-1)}\oplus V_{(1,1,0,\dots,0,-2)}\oplus V_{(2,0,\dots,0,-1,-1)}\\
 =3V_{(1,0,\dots,0,-1)}\oplus V_{(1,1,0,\dots,0,-2)}\oplus 2V_{(2,0,\dots,0,-1,-1)}
 \end{array}$
 & 
 $\begin{array}{l}
 \wedge^2(V)\otimes S^2(V^*)\oplus \wedge^2(V\otimes V^*)\\
 =V_{(1,0,\dots,0,-1)}\oplus V_{(1,1,0,\dots,0,-2)} \\
 \quad \oplus 2V_{(1,0,\dots,0,-1)}\oplus V_{(1,1,0,\dots,0,-2)}\oplus V_{(2,0,\dots,0,-1,-1)}\\
 =3V_{(1,0,\dots,0,-1)}\oplus 2V_{(1,1,0,\dots,0,-2)}\oplus V_{(2,0,\dots,0,-1,-1)}  
 \end{array}$\\
 \hline
 -2 &  
 $\begin{array}{l}
 (V \otimes V^*)\otimes \wedge^2(V^*)\\
 =2V_{(0,\dots,0,-1,-1)}\oplus V_{(1,0,\dots,0,-1,-1,-1)}\\
 \quad \oplus V_{(0,\dots,0,-2)}\oplus V_{(1,0,\dots,0,-1,-2)}
 \end{array}$
 & 
 $\begin{array}{l}
 (V \otimes V^*)\otimes S^2(V^*)\\
 =V_{(0,\dots,0,-1,-1)}\oplus 2V_{(0,\dots,0,-2)}\\
 \quad \oplus V_{(1,0,\dots,0,-1,-2)}\oplus V_{(1,0,\dots,0,-3)} 
 \end{array}$\\
 \hline
 -4 & $S^2(\wedge^2(V^*))=V_{(0,\dots,0,-1,-1,-1,-1)}\oplus V_{(0,\dots,0,-2,-2)}$ & $S^2(S^2(V^*))=V_{(0,\dots,0,-2,-2)}\oplus V_{(0,\dots,0,-4)}$\\
 \hline
\end{tabular}
\label{table:2vertices}
\end{align}

Recall that we assume here $g\geq 6$. Then Theorem \ref{thm:main cohom GC vanishing} states that the weight 2 cohomology of $\GC_{(g),1}$ is concentrated in the critical degree $k^L_{crit}(2,H)$. By inspection, only the graphs of type $D$ above live in the critical degree, while graphs of type $B$ and $C$ live in degree $k^L_{crit}(2,H)-1$ and graphs of type $A$ live in degree $k^L_{crit}(2,H)-2$.
We can hence compute the representation content of $\gr^2H\left(\GC_{(g),1}\right)$ by adding \eqref{equ:triv rep} and \eqref{table:2vertices}, and quotienting by \eqref{table:1vertex} and \eqref{table:2vertex2edge}.




The following table then decomposes $\gr^2H\left(\GC_{(g),1}\right)$ in terms of irreducible representations of $GL_g$. 

\begin{align}
\begin{tabular}{|l|l|l|}
 \hline
 H & m even & m odd \\
 \hline
 4 & $V_{(2,2,0,\dots,0)}$ & $V_{(2,2,0,\dots,0)}$\\
 \hline
 2 & $V_{(2,0,\dots,0)}\oplus V_{(2,1,0,\dots,0,-1)}$ 
 & $V_{(1,1,0,\dots,0)}\oplus V_{(2,1,0,\dots,0,-1)}$\\ 
 \hline
 0 & $V_{(1,0,\dots,0,-1)}\oplus V_{(1,1,0,\dots,0,-2)}\oplus V_{(2,0,\dots,0,-1,-1)}$ 
 & $V_{(1,0,\dots,0,-1)}\oplus V_{(1,1,0,\dots,0,-2)}\oplus V_{(2,0,\dots,0,-1,-1)}$\\
 \hline
 -2 & $V_{(0,\dots,0,-1,-1)}\oplus V_{(1,0,\dots,0,-1,-2)}$ 
 & $V_{(0,\dots,0,-2)}\oplus V_{(1,0,\dots,0,-1,-2)}$\\
 \hline
 -4 & $V_{(0,\dots,0,-2,-2)}$ & $V_{(0,\dots,0,-2,-2)}$\\
 \hline
\end{tabular}
\label{table:cohomology}
\end{align}

\subsection{Decomposition of the quadratic relations $R_{(g),1}$}

We can deduce the decomposition of the quadratic relations $R_{(g),1}$ in terms of irreducible representations of $GL_g$ by the following relation
\[
	\gr^2H\left(\GC_{(g),1}\right) = FreeLie^2\left(gr^1 H\left(GC_{(g),1}\right)\right) / R_{(g),1}.
\]

Let us first consider $FreeLie^2\left(gr^1H\left(GC_{(g),1}\right)\right) = \wedge^2\left(gr^1H\left(GC_{(g),1}\right)\right)$. We need to calculate the second exterior power of the space of generators listed in table (\ref{table:generators}). The following table decomposes $\wedge^2\left(gr^1H\left(GC_{(g),1}\right)\right)$ in terms of irreducible representations of $GL_g$.

\begin{align}
\begin{tabular}{|l|l|l|}
 \hline
 H & m even\\
 \hline
 6 & $V_{(3,3,0,\dots,0)}\oplus V_{(5,1,0,\dots,0)}$\\
 \hline
 4 & $\begin{array}{l}
	V_{(2,2,0,\dots,0)}\oplus 2V_{(3,1,0,\dots,0)}\oplus V_{(3,2,0,\dots,0,-1)}\\
	\oplus 2V_{(4,0,\dots,0)}\oplus V_{(4,1,0,\dots,0,-1)}\oplus V_{(5,0,\dots,0,-1)}
	\end{array}$\\
 \hline
 2 & $\begin{array}{l}
	2V_{(1,1,0,\dots,0)}\oplus 2V_{(2,0,\dots,0)}\oplus 3V_{(2,1,0,\dots,0,-1)}\\
	\oplus V_{(2,2,0,\dots,0,-1,-1)}\oplus 4V_{(3,0,\dots,0,-1)}\oplus V_{(3,1,0,\dots,-1,-1)}\\ 
	\oplus V_{(3,1,0,\dots,0,-2)}\oplus 2V_{(4,0,\dots,0,-1,-1)}
	\end{array}$\\
 \hline
 0 & $\begin{array}{l}
 	V_{(0,\dots,0)}\oplus 4V_{(1,0,\dots,0,-1)}\oplus 2V_{(1,1,0,\dots,0,-1,-1)}\\
 	\oplus V_{(1,1,0,\dots,0,-2)} \oplus 5V_{(2,0,\dots,0,-1,-1)}\oplus 2V_{(2,0,\dots,0,-2)}\\
 	\oplus V_{(2,1,\dots,0,-1,-1,-1)} \oplus V_{(2,1,\dots,0,-1,-2)}\oplus 2V_{(3,0,\dots,0,-1,-1,-1)}\\ \oplus V_{(3,0,\dots,0,-1,-2)}
 	\end{array}$\\
 \hline
 -2 & $\begin{array}{l}
 	 4V_{(0,\dots,0,-1,-1)} \oplus 4V_{(1,0,\dots,0,-1,-1,-1)}\oplus 3V_{(1,0,\dots,0,-1,-2)}\\
 	 \oplus V_{(1,1,0,\dots,0,-1,-1,-1,-1)}\oplus V_{(1,1,0,\dots,0,-2,-2)}\oplus V_{(2,0\dots,0,-1,-1,-1,-1)}\\
 	 \oplus 2V_{(2,0,\dots,0,-1,-1,-2)}  
 	\end{array}$\\
 \hline
 -4 & $\begin{array}{l}
 	2V_{(0,\dots,0,-1,-1,-1,-1)}\oplus 2V_{(0,\dots,0,-1,-1,-2)}\oplus V_{(0,\dots,0,-2,-2)}\\
 	\oplus V_{(1,\dots,0,-1,-1,-1,-1,-1)}\oplus  V_{(1,\dots,0,-1,-1,-1,-2)}\oplus V_{(1,\dots,0,-1,-2,-2)}
 	\end{array}$\\
 \hline
 -6 & $V_{(0,\dots,0,-1,-1,-1,-1,-1,-1)}\oplus V_{(0,\dots,0,-1,-1,-2,-2,)}$\\
 \hline
 \hline
 H & m odd \\
 \hline
 6 & $V_{(1,1,1,1,1,1,0,\dots,0)}\oplus V_{(2,2,1,1,0,\dots,0)}$\\
 \hline
 4 &  $\begin{array}{l}
 	2V_{(1,1,1,1,0,\dots,0)}\oplus 2V_{(2,1,1,0,\dots,0)}\oplus V_{(2,2,0,\dots,0)}\\
 	\oplus V_{(1,1,1,1,1,\dots,0,-1)}\oplus  V_{(2,1,1,1,\dots,0,-1)}\oplus V_{(2,2,1,\dots,0,-1)}
 	\end{array}$\\
 \hline
 2 & $\begin{array}{l}
 	 4V_{(1,1,0,\dots,0)} \oplus 4V_{(1,1,1,0,\dots,0,-1)}\oplus 3V_{(2,1,0,\dots,0,-1)}\\
 	 \oplus V_{(1,1,1,1,0,\dots,0,-1,-1)}\oplus V_{(2,2,0,\dots,0,-1,-1)}\oplus V_{(1,1,1,1,0\dots,0,-2)}\\
 	 \oplus 2V_{(2,1,1,0,\dots,0,-2)}  
 	\end{array}$\\
 \hline
 0 & $\begin{array}{l}
 	V_{(0,\dots,0)}\oplus 4V_{(1,0,\dots,0,-1)}\oplus 2V_{(1,1,0,\dots,0,-1,-1)}\\
 	\oplus V_{(2,0,\dots,0,-1,-1)}\oplus 5V_{(1,1,0,\dots,0,-2)}\oplus 2V_{(2,0,\dots,0,-2)}\\
 	\oplus V_{(1,1,1,\dots,0,-1,-2)}\oplus V_{(2,1,\dots,0,-1,-2)}\oplus 2V_{(1,1,1,0,\dots,0,-3)}\\
 	\oplus V_{(1,2,0,\dots,0,-3)}
 	\end{array}$\\
 \hline
 -2 & $\begin{array}{l}
	2V_{(0,\dots,0,-1,-1)}\oplus 2V_{(0,\dots,0,-2)}\oplus 3V_{(1,0,\dots,0,-1,-2)}\\
	\oplus V_{(1,1,0,\dots,0,-2,-2)}\oplus 4V_{(1,0,\dots,0,-3)}\oplus V_{(1,1,0,\dots,-1,-3)}\\ 
	\oplus V_{(2,0,\dots,0,-1,-3)}\oplus 2V_{(1,1,0,\dots,0,-4)}
	\end{array}$\\
 \hline
 -4 & $\begin{array}{l}
	V_{(0,\dots,0,-2,-2)}\oplus 2V_{(0,\dots,0,-1,-3)}\oplus V_{(1,0,\dots,0,-2,-3)}\\
	\oplus 2V_{(0,\dots,0,-4)}\oplus V_{(1,0,\dots,0,-1,-4)}\oplus V_{(1,0,\dots,0,-5)}
	\end{array}$\\
 \hline
 -6 & $V_{(0,\dots,0,-3,-3)}\oplus V_{(0,\dots,0,-1,-5)} $\\
 \hline
\end{tabular}
\label{table:freelie}
\end{align}

$R_{(g),1}$ consists of the $\GL_g$-representations of table (\ref{table:freelie}) quotient those of table (\ref{table:cohomology}). 
The following table hence decomposes $R_{(g),1}$ in terms of irreducible representations of $\GL_g$. 

\begin{align}
\begin{tabular}{|l|l|l|}
 \hline
 H & m even\\
 \hline
 6 & $V_{(3,3,0,\dots,0)}\oplus V_{(5,1,0,\dots,0)}$\\
 \hline
 4 & $\begin{array}{l}
	2V_{(3,1,0,\dots,0)}\oplus V_{(3,2,0,\dots,0,-1)}\oplus 2V_{(4,0,\dots,0)}\\
	\oplus V_{(4,1,0,\dots,0,-1)}\oplus V_{(5,0,\dots,0,-1)}
	\end{array}$\\
 \hline
 2 & $\begin{array}{l}
	2V_{(1,1,0,\dots,0)}\oplus V_{(2,0,\dots,0)}\oplus 2V_{(2,1,0,\dots,0,-1)}\\
	\oplus V_{(2,2,0,\dots,0,-1,-1)}\oplus 4V_{(3,0,\dots,0,-1)}\oplus V_{(3,1,0,\dots,-1,-1)}\\ 
	\oplus V_{(3,1,0,\dots,0,-2)}\oplus 2V_{(4,0,\dots,0,-1,-1)}
	\end{array}$\\
 \hline
 0 & $\begin{array}{l}
 	V_{(0,\dots,0)}\oplus 3V_{(1,0,\dots,0,-1)}\oplus 2V_{(1,1,0,\dots,0,-1,-1)}\\
 	\oplus 4V_{(2,0,\dots,0,-1,-1)} \oplus 2V_{(2,0,\dots,0,-2)}\oplus V_{(2,1,\dots,0,-1,-1,-1)}\\
 	\oplus V_{(2,1,\dots,0,-1,-2)}\oplus 2V_{(3,0,\dots,0,-1,-1,-1)}\oplus V_{(3,0,\dots,0,-1,-2)}
 	\end{array}$\\
 \hline
 -2 & $\begin{array}{l}
 	 3V_{(0,\dots,0,-1,-1)} \oplus 4V_{(1,0,\dots,0,-1,-1,-1)}\oplus 2V_{(1,0,\dots,0,-1,-2)}\\
 	 \oplus V_{(1,1,0,\dots,0,-1,-1,-1,-1)}\oplus V_{(1,1,0,\dots,0,-2,-2)}\oplus V_{(2,0\dots,0,-1,-1,-1,-1)}\\
 	 \oplus 2V_{(2,0,\dots,0,-1,-1,-2)}
 	\end{array}$\\
 \hline
 -4 & $\begin{array}{l}
 	2V_{(0,\dots,0,-1,-1,-1,-1)}\oplus 2V_{(0,\dots,0,-1,-1,-2)}\oplus V_{(1,\dots,0,-1,-1,-1,-1,-1)}\\
 	\oplus  V_{(1,\dots,0,-1,-1,-1,-2)}\oplus V_{(1,\dots,0,-1,-2,-2)}
 	\end{array}$\\
 \hline
 -6 & $V_{(0,\dots,0,-1,-1,-1,-1,-1,-1)}\oplus V_{(0,\dots,0,-1,-1,-2,-2,)}$\\
 \hline
 \hline
 H & m odd \\
 \hline
 6 & $V_{(1,1,1,1,1,1,0,\dots,0)}\oplus V_{(2,2,1,1,0,\dots,0)}$\\
 \hline
 4 &  $\begin{array}{l}
 	2V_{(1,1,1,1,0,\dots,0)}\oplus 2V_{(2,1,1,0,\dots,0)}\oplus V_{(1,1,1,1,1,\dots,0,-1)}\\
 	\oplus  V_{(2,1,1,1,\dots,0,-1)}\oplus V_{(2,2,1,\dots,0,-1)}
 	\end{array}$\\
 \hline
 2 & $\begin{array}{l}
 	 3V_{(1,1,0,\dots,0)} \oplus 4V_{(1,1,1,0,\dots,0,-1)}\oplus 2V_{(2,1,0,\dots,0,-1)}\\
 	 \oplus V_{(1,1,1,1,0,\dots,0,-1,-1)}\oplus V_{(2,2,0,\dots,0,-1,-1)}\oplus V_{(1,1,1,1,0\dots,0,-2)}\\
 	 \oplus 2V_{(2,1,1,0,\dots,0,-2)}
 	\end{array}$\\
 \hline
 0 & $\begin{array}{l}
 	V_{(0,\dots,0)}\oplus 3V_{(1,0,\dots,0,-1)}\oplus 2V_{(1,1,0,\dots,0,-1,-1)}\\
 	\oplus 4V_{(1,1,0,\dots,0,-2)}\oplus 2V_{(2,0,\dots,0,-2)}\oplus V_{(1,1,1,\dots,0,-1,-2)}\\
 	\oplus V_{(2,1,\dots,0,-1,-2)}\oplus 2V_{(1,1,1,0,\dots,0,-3)}\oplus V_{(1,2,0,\dots,0,-3)}
 	\end{array}$\\
 \hline
 -2 & $\begin{array}{l}
	2V_{(0,\dots,0,-1,-1)}\oplus V_{(0,\dots,0,-2)}\oplus 2V_{(1,0,\dots,0,-1,-2)}\\
	\oplus V_{(1,1,0,\dots,0,-2,-2)}\oplus 4V_{(1,0,\dots,0,-3)}\oplus V_{(1,1,0,\dots,-1,-3)}\\ 
	\oplus V_{(2,0,\dots,0,-1,-3)}\oplus 2V_{(1,1,0,\dots,0,-4)}
	\end{array}$\\
 \hline
 -4 & $\begin{array}{l}
	2V_{(0,\dots,0,-1,-3)}\oplus V_{(1,0,\dots,0,-2,-3)}\oplus 2V_{(0,\dots,0,-4)}\\
	\oplus V_{(1,0,\dots,0,-1,-4)}\oplus V_{(1,0,\dots,0,-5)}
	\end{array}$\\
 \hline
 -6 & $V_{(0,\dots,0,-3,-3)}\oplus V_{(0,\dots,0,-1,-5)} $\\
 \hline
\end{tabular}
\label{table:freelie}
\end{align} 

\section{Lower bound and proof of part (i) of Theorem \ref{thm:main cohom GC vanishing}}

\newcommand{\ppICG}{\mathsf{ppICG}}
\newcommand{\pICG}{\mathsf{pICG}}
\newcommand{\ICG}{\mathsf{ICG}}
\subsection{Kontsevich graph operad and degree bound on hairy graph complexes}
We consider the Kontsevich graph operad $\Graphs_n$ \cite{K2} and three dg sub-symmetric sequences
\[
  \ppICG_n[1]\subset \pICG_n[1] \subset \ICG_n[1] \subset \Graphs_n.
\]
Here $\Graphs_n(r)$ is a dg vector space spanned by graphs with $r$ numbered external and an arbitrary number of internal vertices.
The subspace $\ICG[1] \subset \Graphs_n$ is spanned by graphs that are internally connected, i.e., that are connected after removing all external vertices.
The sub-symmetric sequence $\pICG \subset \ICG$ is spanned by graphs all of whose external vertices have valence $\geq 1$. The sub-symmetric sequence $\ppICG \subset \pICG$ is spanned by graphs all of whose external vertices have valence $=1$. In particular, the resulting complex $\ppICG(r)$ is a complex of graphs with $r$ external legs (hairs) that are labeled by the numbers $1,\dots,r$. Let us give a few examples:
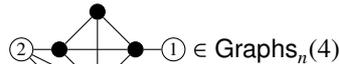
\begin{align*}
	\begin{tikzpicture}[baseline=-.65ex]
		\node[int] (v1) at (0:.5) {};
		\node[int] (v2) at (90:.5) {};
		\node[int] (v3) at (180:.5) {};
		\node[int] (v4) at (270:.5) {};
		\node[ext] (w1) at (0:1) {$\scriptstyle 1$};
		\node[ext] (w3) at (180:1) {$\scriptstyle 2$};
		\node[ext] (w4) at (270:1) {$\scriptstyle 3$};
		\node[ext] (w5) at (-45:1) {$\scriptstyle 4$};
		\draw (v1) edge (v2) edge (v3) edge (v4) edge (w1)
		(v2) edge (v3) edge (v4) 
		(v3) edge (v4) edge (w3)
		(v4) edge (w4)
		(w3) edge (v4) edge (w4);
	  \end{tikzpicture}
	  &\in \Graphs_n(4)
	&
	
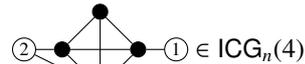
\begin{tikzpicture}[baseline=-.65ex]
		\node[int] (v1) at (0:.5) {};
		\node[int] (v2) at (90:.5) {};
		\node[int] (v3) at (180:.5) {};
		\node[int] (v4) at (270:.5) {};
		\node[ext] (w1) at (0:1) {$\scriptstyle 1$};
		\node[ext] (w3) at (180:1) {$\scriptstyle 2$};
		\node[ext] (w4) at (270:1) {$\scriptstyle 3$};
		\node[ext] (w5) at (-45:1) {$\scriptstyle 4$};
		\draw (v1) edge (v2) edge (v3) edge (v4) edge (w1)
		(v2) edge (v3) edge (v4) 
		(v3) edge (v4) edge (w3)
		(v4) edge (w4)
		(w3) edge (v4);
	  \end{tikzpicture}
	  &\in \ICG_n(4)
	\\
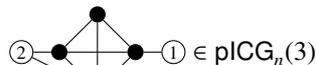
\begin{tikzpicture}[baseline=-.65ex]
  \node[int] (v1) at (0:.5) {};
  \node[int] (v2) at (90:.5) {};
  \node[int] (v3) at (180:.5) {};
  \node[int] (v4) at (270:.5) {};
  \node[ext] (w1) at (0:1) {$\scriptstyle 1$};
  \node[ext] (w3) at (180:1) {$\scriptstyle 2$};
  \node[ext] (w4) at (270:1) {$\scriptstyle 3$};
  \draw (v1) edge (v2) edge (v3) edge (v4) edge (w1)
  (v2) edge (v3) edge (v4) 
  (v3) edge (v4) edge (w3)
  (v4) edge (w4)
  (w3) edge (v4);
\end{tikzpicture}
&\in \pICG_n(3)
	&
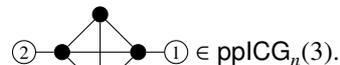
\begin{tikzpicture}[baseline=-.65ex]
  \node[int] (v1) at (0:.5) {};
  \node[int] (v2) at (90:.5) {};
  \node[int] (v3) at (180:.5) {};
  \node[int] (v4) at (270:.5) {};
  \node[ext] (w1) at (0:1) {$\scriptstyle 1$};
  \node[ext] (w3) at (180:1) {$\scriptstyle 2$};
  \node[ext] (w4) at (270:1) {$\scriptstyle 3$};
  \draw (v1) edge (v2) edge (v3) edge (v4) edge (w1)
  (v2) edge (v3) edge (v4) 
  (v3) edge (v4) edge (w3)
  (v4) edge (w4);
\end{tikzpicture}
&\in \ppICG_n(3).
\end{align*}

The degree shifts by one are conventional, to agree with the literature.
To be explicit, the degree of a graph $\Gamma$ in either of $\ppICG_n$, $\pICG_n$ or $\ICG_n$ is given by 
\[
\#(\text{internal vertices})\cdot n - \#(\text{edges})\cdot (n-1) +1.
\]

Furthermore, one defines the \emph{complexity} grading on each of the four dg symmetric sequences above, by assigning a graph $\Gamma$ the complexity 
\[
C=  \#(\text{edges}) - \#(\text{internal vertices}).
\]
In other words the complexity is the loop order of the graph formed by fusing all the external vertices into one vertex.
Or equivalently, the loop order of a graph $\Gamma$ with $r$ external vertices and complexity $C$ is 
\[
L = C+r-1.  
\]
We denote the direct summand of complexity $C$ of any of the symmetric sequences above by the notation $\gr^C (-)$.

The differential $\delta$ on each of the graph complexes above is given by vertex splitting, producing one additional internal vertex and leaving the complexity unchanged.
\begin{align*}
	\delta
	\begin{tikzpicture}[baseline=-.65ex]
		\node[ext] (v) {$\scriptstyle j$};
		\draw (v) edge +(-.5,.5) edge +(-.2,.5) edge +(.2,.5) edge +(.5,.5);
		\end{tikzpicture}
		&=\sum
	\begin{tikzpicture}[baseline=-.65ex]
	\node[ext] (v) at (0,0) {$\scriptstyle j$};
	\node[int](w) at (0,.5) {};
	\draw (v) edge +(-.5,.5) edge +(.5,.5) edge (w) (w) edge +(-.2,.5) edge +(.2,.5);
	\end{tikzpicture}
	&
	\delta
	\begin{tikzpicture}[baseline=-.65ex]
		\node[int] (v) {};
		\draw (v) edge +(-.5,.5) edge +(-.2,.5) edge +(.2,.5) edge +(.5,.5);
		\end{tikzpicture}
		&=\sum
	\begin{tikzpicture}[baseline=-.65ex]
	\node[int] (v) at (0,0) {};
	\node[int](w) at (0,.5) {};
	\draw (v) edge +(-.5,.5) edge +(.5,.5) edge (w) (w) edge +(-.2,.5) edge +(.2,.5);
	\end{tikzpicture}
\end{align*}

One then has the following result:
\begin{thm}[{\v Severa-Willwacher \cite[Proposition 2]{SW}}]
  \label{thm:SW}
The summand of complexity $C$ of the cohomology of $\ICG_n$ is concentrated in degree $-(n-2)C$, i.e., 
\[
  \gr^C H^k( \ICG_n) = 0  
\]
for $k\neq -(n-2)C$.
\end{thm}
\begin{proof}
To be precise, \cite[Proposition 2]{SW} shows the slightly different statement that $H(\ICG_2(r))$ can be identified with the 
Drinfeld-Kohno Lie algebra $\alg t_2(r)$, with the generators $t_{ij}$ of the Drinfeld-Kohno Lie algebra corresponding to graphs with no internal vertices, and an edge between external vertices $i$ and $j$. The proof given in \cite{SW} is written for $n=2$, but is independent of $n$, and shows likewise that $H(\ICG_n(r))$ is identified with the $n$-Drinfeld-Kohno Lie algebra $\alg t_n(r)$.
Again the generators $t_{ij}$ of $\alg t_n(r)$ correspond to graphs with a single edge between external vertices $i$ and $j$, and no internal vertex.
Since these generators hence have complexity $+1$ and degree $-(n-2)$ our theorem follows. 
\end{proof}
\begin{cor}\label{cor:SW}
The summand of complexity $C$ of the cohomology of $\pICG_n$ is concentrated in degree $-(n-2)C$.
\end{cor}
\begin{proof}
  $\pICG_n\subset \ICG_n$ is a direct summand, a one-sided inverse $\ICG_n\twoheadrightarrow \pICG_n$ of the inclusion is given by setting to zero graphs with zero-valent external vertices. Hence the result of Theorem \ref{thm:SW} implies the corollary.
\end{proof}

\begin{cor}
\label{lemma:lower bound ppicg}
  The summand of complexity $C$ of the cohomology of $\ppICG_n$ is concentrated in degrees $\geq -(n-2)C$.
\end{cor}
This result is essentially known, but hard to reference in the literature.
It is a formal consequence of the observation that $\pICG_n$ is quasi-cofree as a right operadic $\Com$-module, with the space of cogenerators identified with $\ppICG_n$. For convenience we provide a self-contained elementary proof here, without reference to operadic modules.
\begin{proof}
We fix a complexity $C$ and consider the finite dimensional dg vector spaces
\[
 \gr^C \ppICG_n(r)
\]
for $r=1,2,\dots$. Since the complexity of a graph in $\ppICG_n(r)$ is at least $r-1$, there are only finitely many $r$ for which the complex above is non-zero. 
If the cohomology of all $\gr^C \ppICG_n(r)$ is zero we are done trivially.
Otherwise, there is a non-trivial cohomology class of lowest degree $k_{0}$, say $x\in H^{k_{0}}(\gr^C \ppICG_n(r_0))$.

Next consider the finite dimensional dg vector space
\[
 \gr^C \pICG_n(r_0).
\]
We filter this vector space by the sum of valences of the external vertices of graphs minus $r_0$. We call this number the \emph{defect}, and denote the pieces of defect $D$ of the corresponding associated graded by 
\[
 \gr_{def}^D \gr^C \pICG_n(r_0).
\]
In particular, there are no graphs of negative defect and we have that 
\begin{equation}\label{equ:grdef0}
 \gr_{def}^0 \gr^C \pICG_n(r_0) = \gr^C \ppICG_n(r_0).
\end{equation}
The differential can only leave constant or decrease the defect. We consider the spectral sequence associated to the filtration by defect. It converges to the cohomology of $\gr^C \pICG_n(r_0)$ by finite dimensionality of our complex, and by Corollary \ref{cor:SW} the cohomology is concentrated in degree $-C(n-2)$.

We consider the first page, i.e., the associated graded complexes $\gr_{def}^D \gr^C \pICG_n(r_0)$.
Combinatorially, the differential on the associated graded splits internal vertices, but not external vertices, since that would reduce the defect. 
The crucial observation is now that the complexes $\gr_{def}^D \gr^C \pICG_n(r_0)$ are identified with direct sums of direct summands of $\gr^C \ppICG_n(r_0+D)$.
More precisely, we have the direct sum decomposition of complexes
\[
  \gr_{def}^D \gr^C \pICG_n(r_0)  
  = \bigoplus_{\nu_1,\dots,\nu_{r_0}\geq 1\atop \nu_1+\cdots+\nu_{r_0}=D+r_0} V_{\nu_1,\cdots,\nu_{r_0}},
\]
with the subcomplex $V_{\nu_1,\cdots,\nu_{r_0}}$ spanned by graphs $\Gamma$ whose external vertex $j$ has valence $\nu_j$ for $j=1,\dots,r_0$.
Then we have that 
\[
  V_{\nu_1,\cdots,\nu_{r_0}} \cong \left(\gr^C \ppICG_n(r_0+D) \right)_{S_{\nu_1}\times \cdots \times S_{\nu_{r_0}}}.
\]
Since we assumed that the cohomology of $\gr^C \ppICG_n(r_0+D)$ is concentrated in degrees $\geq k_0$, we conclude that the cohomology of $\gr_{def}^D \gr^C \pICG_n(r_0)$ is concentrated in degrees $\geq k_0$ as well.

Now, using \eqref{equ:grdef0}, our non-trivial class $x\in H^{k_{0}}(\gr^C \ppICG_n(r_0))$ defines a non-trivial class of lowest degree $k_0$ and of defect 0 on the second page of our spectral sequence converging to $H(\pICG_n(r_0))$.
Hence it cannot be exact on any page of the spectral sequence by degree reasons.
Furthermore the higher differentials in the spectral sequence all decrease the defect, hence our class must be closed on all pages.
Hence it survives in cohomology and we conclude from Corollary \ref{cor:SW} that $k_0=-C(n-2)$.
\end{proof}

\subsection{Expressing $\GC_{(g),1}$ in terms of $\ppICG_n$}
\label{sec: GCg}

 Let $V_g\cong H^m(W_{g,1})\cong H_{m+1}(W_{g,1})$ and $V_g^* \cong H^{m+1}(W_{g,1})\cong H_{m}(W_{g,1})$ correspond to the defining and adjoint representations of $GL_g$ respectively. Every vertex of a graph in $\GC_{(g),1}$ is decorated by elements of $S\left(\bar{H}(W_{g,1})\right)\cong S\left(V_g\oplus V_g^*\right)$. We can express the graph complex $\GC_{(g),1}$ in terms of $\ppICG_n$, for $n=2m+1$:
\[
	\GC_{(g),1}\cong\prod_r\left(\ppICG_n(r)\otimes \left( V_g[m-n+1]\oplus V_g^*[m-n+2]\right)^{\otimes r}\right)_{S_r}.
\]

Let $\Gamma$ denote a graph in $\GC_{(g),1}$ and $\tilde{\Gamma}$ the corresponding graph in $\ppICG_n$ and let $\tilde{k}$ be the cohomological degree of $\tilde{\Gamma}$. Furthermore, let $a$ denote the number of decorations of $\Gamma$ in $V_g$ and b the number of decorations in $V_g^*$. The degree $k$ of $\Gamma$ is then given as
\[
	k=\tilde{k} +(a+b)(n-1)-am-b(m+1).
\]

Let $v$ and $e$ denote the number of vertices and edges of $\Gamma$ and $\tilde{e}$ the number of edges of $\tilde{\Gamma}$. The complexity $C$ is then defined as $C=\tilde{e}-v= e-v+a+b$. 

Now we can use the lower degree bound for the cohomology of $\ppICG_n$ as shown in Lemma \ref{lemma:lower bound ppicg} to derive a lower bound for the degree of the cohomology of $\GC_{(g),1}$:
\[
\begin{aligned}
	k &=\tilde{k} +(a+b)(n-1)-am-b(m+1)\\
	&\geq -(n-2)C+(a+b)(n-1)-am-b(m+1)\\
	&=-(n-2)C+am+b(m-1)
\end{aligned}
\]
Since we restrict to the weight $W$ and imbalance $H$ part we can espress $a$, $b$, and $C$ in terms of the imbalance $H=a-b$ and the weight 
\[
	W= 2(e-v)+D=2(e-v)+a+b=2C-(a+b):
\]
\[
\begin{aligned}
	a =&\frac{1}{2}(-W+H)+C\\
	b =&\frac{1}{2}(-W-H)+C
\end{aligned}
\]
Finally, we derive the lower bound for the degree of $\Gamma$:
\[
\begin{aligned}
	k&\geq -(2m-1)C +m\frac{1}{2}(-W+H)+(m-1)\frac{1}{2}(-W-H)+(2m-1)C\\
	&=-\frac{1}{2}(2m-1)W+\frac{H}{2}\\
	&=k_{crit}^L
\end{aligned}
\]

\subsection{Proof of part (i) of Theorem \ref{thm:main cohom GC vanishing}}
From section \ref{sec: GCg} we know that 
\[ 
	H\left(gr^{W,H}\GC_{(g),1}\right)=0 ~,~\text{for}~ k<-k_{crit}^L(W,H).
\]
Since the differential is compatible with the bigrading we can deduce the lower bound in part (i) of Theorem \ref{thm:main cohom GC vanishing}.

\subsection{The $\osp_{g,1}^{nil}$ part of $\GCex_{(g),1}$}
\label{sec: ospnil}
Let us also consider the $\osp_{g,1}^{nil}$ part of $\GCex_{(g),1}$. An element of $\osp_{g,1}^{nil}$ maps a decoration in $H_{m}(W_{g,1})$ to a decoration in $H_{m+1}(W_{g,1})$.
It therefore contributes to the imbalance but not to the weight, i.e. $H=-2$ and $W=0$, and has cohomological degree $-1=k_{crit}^L(0,-2)$.

%
\section{Upper bound and proof of part (ii) of Theorem \ref{thm:main cohom GC vanishing}} \label{sec:main cohom GC vanishing proof 3}

\subsection{Graph complex with two-colored edges -- recollection from \cite{BM, FNW}}\label{sec:CGamma}

Let $\Gamma$ be a graph with vertex set $\{1,\dots,N\}$ and $k$ ordered edges.
We shall need an auxiliary dg vector space $C_\Gamma$ spanned by certain colorings of the edges of $\Gamma$. This dg vector space has appeared first in \cite{BM}. We shall repeat here the definition from \cite{FNW}.

To this end, let $C$ be the two-dimensional acyclic complex
\[
C = (\Q[1] \to \Q),
\]
and consider the complex 
\[
C^{\otimes k}.  
\]
We interpret the natural basis elements of $C^{\otimes k}$ as assignments of dash-patterns to edges of $\Gamma$, with $\Q[1]$ corresponding to a solid edge, and $\Q$ to a dashed edge.
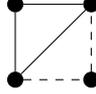
\[
\begin{tikzpicture}
  \node[int] (v1) at (0,0) {};
  \node[int] (v2) at (1,0) {};
  \node[int] (v3) at (0,1) {};
  \node[int] (v4) at (1,1) {};
  \draw (v1) edge[dashed] (v2) edge (v3) edge (v4)
  (v2) edge[dashed] (v4)
  (v3) edge (v4);
\end{tikzpicture}
\]

The differential acts by summing over edges, and replacing a solid edge by a dashed one, rendering $C^{\otimes k}$ is acyclic if $k> 0$.

Let $I_{disc}\subset C^{\otimes k}$ be the subcomplex spanned by all such graphs for which the subgraph consisting of the vertices and the solid edges is not connected and define 
\[
C_\Gamma = C^{\otimes k}/I_{disc}
\] 
to be the quotient complex obtained by setting such solid-disconnected graphs to zero.
Then $C_\Gamma$ is concentrated in cohomological degrees $-k,\dots, 1-N$.

\begin{lemma}[Lemma 8.8 of \cite{BM} or Lemma 35 of \cite{FNW}]
\label{lem:CGamma}
The cohomology of $C_\Gamma$ is concentrated in the top degree $1-N$.
\end{lemma}
The cohomology is hence given by linear combinations of solid-trees, modulo the differential of solid-one-loop graphs.

\subsection{$\GC_{(g),1}$ and its graded dual $\G_{(g),1}$}

Instead of the cohomology of $\GC_{(g),1}$ we can study the cohomology of its graded dual $\left(\GC_{(g),1}\right)^c=\G_{(g),1}$. We consider the weight $W$ and imbalance $H$ part $gr^{(W,H)}H\left(\G_{(g),1}\right)$ of its cohomology. Since the bigrading is compatible with the differential this is equivalent to $H\left(gr^{(W,H)}\G_{(g),1}\right)$. As before let $V_g\cong H^m(W_{g,1})$ and $V_g^*\cong H^{m+1}(W_{g,1})$ correspond to the defining and dual representations of $\GL_g$ respectively. Every vertex of a graph in $\G_{(g),1}$ is decorated by elements of $S\left(\bar{H}(W_{g,1})\right)\cong S\left(V_g\oplus V_g^*\right)$.

 
\subsection{Stabilisation for large $g$}
\label{sec:stabil_G}

In order to apply the invariant theory for $\GL_g$ we study the auxiliary graph complex

\[
G^g_{M,N} :=\left(\G_{(g),1}\otimes V_g^M\otimes \left(V_g^*\right)^N\right)^{\GL_g}
\]
and its graded version
\[
gr^{(W,H)} G^g_{M,N} =gr^{(W,H)}\left(\G_{(g),1}\otimes V_g^M\otimes \left(V_g^*\right)^N\right)^{\GL_g}=\left(gr^{(W,H)}\G_{(g),1}\otimes V_g^M\otimes \left(V_g^*\right)^N\right)^{\GL_g}.
\]

We have natural inclusions of groups $\GL_{g-1}\subset \GL_g$, and $\GL_{g-1}$-equivariant isomorphisms 
\begin{align*}
V_g &\cong V_{g-1} \oplus \Q & V_g^* &\cong V_{g-1}^* \oplus \Q.
\end{align*}
Hence we have sequence of natural projection maps 
\begin{equation}\label{equ:Ggtower}
	\cdots \to G^{g+1}_{M,N} \to G^{g}_{M,N}\to G^{g-1}_{M,N} \to \cdots\, .
\end{equation}

\begin{lemdef}
\label{lemdef:Ginfty}
	The sequence above stabilizes in each finite weight.
	That is, for each $M,N,W,H, g$ with $H=N-M$ and $2g\geq W+M+N+2$ the map 
	\[
		gr^{(W,H)}G_{M,N}^{g+1} \to gr^{(W,H)}G_{M,N}^{g}
	\]
	is an isomorphism.
	We define the bigraded dg vector space 
	\[
		G_{M,N}^{\infty} = \bigoplus_{W,H} 	gr^{(W,H)}G_{M,N}^{\infty}
	\]
	such that 
	\[
		gr^{(W,H)}G_{M,N}^{\infty} := gr^{(W,H)}G_{M,N}^{g}
	\]
	for any $g\geq \frac 12(W+M+N+2)$. In other words, $G_{M,N}^{\infty}$ is the weight-wise limit of the tower \eqref{equ:Ggtower}.
\end{lemdef}
Note that in particular we obtain natural maps 
\[
	G_{M,N}^{\infty} \to G_{M,N}^{g}
\]
for every $g$.

\begin{proof}
Let us consider a graph in $\G_{(g),1}$. We denote the graph without the decorations in $S(V_g\oplus V_g^*)$ of its vertices as the core of the corresponding graph. As a representation of $\GL_g$, $\G_{(g),1}$ corresponds to a direct sum over isomorphism classes of cores $\Gamma$
\[
	\G_{(g),1}\cong\bigoplus_{\Gamma}U_\Gamma
\]
with
\[
	U_\Gamma\cong\left(\left(\bigotimes_{v\in V\Gamma}S^{\geq k(v)}\left(V_g\oplus V_g^*\right)\right)\otimes W_\Gamma \right)_{\Aut(\Gamma)}.
\]
Here $\Aut(\Gamma)$ is the symmetry group of the core $\Gamma$ and $W_{\Gamma}$ is a one dimensional representation of $\Aut(\Gamma)$, on which $\GL_g$ acts trivially. The number $k(v)$ is the minimal number of decorations on the core vertex $v$ needed to satisfy the trivalence assumption, i.e., if $v$ is $j$-valent in the core, then $k(v)=\max(0,3-j)$.
Expanding the symmetric products, we furthermore get 
\[
	U_\Gamma\cong\left(\bigoplus_{\alpha,\beta: V\Gamma \to \Z_{\geq 0} \atop \alpha(v)+\beta(v) \geq k(v)} \left((V_g)^{\otimes \sum_v \alpha(v)}\otimes \left(V_g^*\right)^{\otimes \sum_v \beta(v)}\right)_{\prod_v (S_{\alpha(v)}\times S_{\beta(v)})} \otimes W_\Gamma \right)_{\Aut(\Gamma)}.
\]

Accordingly $G_{M,N}^{g}$ splits into a direct sum of subspaces of the form
\begin{align*}
	&\left(\left( \bigoplus_{\alpha,\beta: V\Gamma \to \Z_{\geq 0} \atop \alpha(v)+\beta(v) \geq k(v)} \left((V_g)^{\otimes \sum_v \alpha(v)}\otimes \left(V_g^*\right)^{\otimes \sum_v \beta(v)}\right)_{\prod_v (S_{\alpha(v)}\times S_{\beta(v)})}
	\otimes W_\Gamma \right)_{Aut(\Gamma)}
	\otimes V_g^{\otimes M}\otimes \left(V_g^*\right)^{\otimes N}\right)^{\GL_g}\\
	&\cong
	\left(\left( \bigoplus_{\alpha,\beta: V\Gamma \to \Z_{\geq 0} \atop \alpha(v)+\beta(v) \geq k(v)} \left((V_g)^{\otimes \sum_v \alpha(v)}\otimes \left(V_g^*\right)^{\otimes \sum_v \beta(v)}
	\otimes V_g^{\otimes M}\otimes \left(V_g^*\right)^{\otimes N}
	\right)^{\GL_g}
	\right)_{\prod_v (S_{\alpha(v)}\times S_{\beta(v)})}
	\otimes W_\Gamma
	\right)_{Aut(\Gamma)}
	.
\end{align*}
In the last step we used that taking the coinvariants with respect to the finite symmetry group commutes with taking the invariants with respect to $\GL_g$ and that $\GL_g$ acts trivially on $W_\Gamma$. 

Let $a=\sum_{v\in V\Gamma} \alpha(v)$ denote the number of decorations in $V_g$ and $b=\sum_{v\in V\Gamma} \beta(v)$ the number of decorations in $V_g^*$ of the graph. Finally, we consider the $\GL_g$ invariants of expressions of the form $V_g^{\otimes A} \otimes \left(V_g^*\right)^{\otimes B}$ with $A=a+M$ and $B=b+N$.

If $A=B\leq g$ the fundamental theorems of invariant theory for $\GL_g$ as stated in Theorem \ref{fft} lead to
\[
	\left(V_g^{\otimes A} \otimes \left(V_g^*\right)^{\otimes B}\right)^{\GL_g}\cong \Q[S_A]
\]
which is equivalent to connecting every $V_g^*$ to a $V_g$ by the the $\frac 12$-diagonal element $\Delta^{\frac 12} =(-1)^m \sum_{i=1}^g \alpha_i \otimes \beta_i$ where $V_g=span\{\alpha_1,\dots,\alpha_g\}$ and $V_g^*=span\{\beta_1,\dots,\beta_g\}$. On the other hand if $A\neq B$ the space is zero. Therefore if $g$ is large enough we can deduce that $gr^{(W,H)}G_{M,N}^{g}$ stabilises, i.e. it is independent of the genus $g$ and if $A\neq B$ it vanishes.

Since we consider connected graphs the numbers of edges $e$ and of vertices $v$ satisfy $e-v\geq -1$, and hence we have $W\geq -2 +a+b$. This corresponds to an upper bound for the number of decorations $D=a+b\leq W+2$.
The stabilization condition of Theorem \ref{fft} is $2g\geq A+B=a+b+M+N$, which is hence satisfied if $2g\geq W+M+N+2$.


\end{proof}

The complexes $G_{M,N}^{g}$ have the following significance:
\begin{lemma}
\label{lem:G_vanishing}
If $H^k\left(\gr^{(W,H)}G_{M,N}^{g}\right)=0$ for fixed $W,H,k,g$ and all $M,N$ such that $M+N\leq W+2$ then 
\begin{equation}\label{equ:lem G_vanishing}
H^k\left(\gr^{(W,H)}\G_{(g),1}\right)=0.
\end{equation}
In particular, if $H^k\left(\gr^{(W,H)}G_{M,N}^{\infty}\right)=0$ for all $M,N$ such that $M+N\leq W+2$, then \eqref{equ:lem G_vanishing} holds for all $g$ such that $g\geq W+2$.
\end{lemma}
\begin{proof}
 By assumption we know that 
 \begin{align*}
 0&=H^k\left(\gr^{(W,H)}G_{M,N}^{g}\right)\\
 &=H^k\left(\left(gr^{(W,H)}\G_{(g),1}\otimes V_g^M\otimes \left(V_g^*\right)^N\right)^{\GL_g}\right)\\
 &=\left(H^k\left(gr^{(W,H)}\G_{(g),1}\right)\otimes V_g^M\otimes \left(V_g^*\right)^N\right)^{\GL_g}
 \end{align*}
 for all $M,N$ with $M+N\leq W+2$. In the last equation we used the fact that we can interchange taking the cohomology with taking the $\GL_g$-invariants. 
 
  Due to the upper bound for the number of decorations $D\leq W+2$ derived in the previous proof we know that $H^k\left(\gr^{(W,H)}\G_{(g),1}\right)$ is a $\GL_g$-representation of order $W+2$. Therefore, the first claim follows from Lemma \ref{lem:rep_gl_g}.
  
  The second statement of the Lemma is simply obtained by combining the first statement with the stabilization result Lemma/Definition \ref{lemdef:Ginfty}.
 \end{proof}

\subsection{Auxiliary graph complex $G_{M,N}^\infty$}
\label{sec:combinatorial_G}
The complex $G_{M,N}^\infty$ defined above has a combinatorial description as a graph complex, as can be seen from the proof of Lemma/Definition \ref{lemdef:Ginfty}.

Concretely, this auxiliary graph complex $G_{M,N}^\infty$ is composed of graphs with two different types of vertices, internal vertices and external vertices.
It has two different types of edges, solid undirected edges connecting internal vertices as well as dashed directed edges. The latter can start at internal or external vertices and point to internal or external vertices. Internal vertices have total valence at least three and external vertices have valence one with one incoming or one outgoing dashed edge.
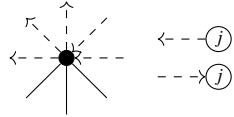
\[
\begin{tikzpicture}
  \node[int] (int) at (0,0) {};
  \draw (int) edge (-135: 0.75) (int) edge (-45: 0.75) (int) edge (-90: 0.75);
  \draw[->] (int) edge[dashed] (180: 0.75) (int) edge[dashed] (135: 0.75) (int) edge[dashed] (90: 0.75);
  \path[->] (45: 0.75) edge[dashed] (int);
  \path[->] (0: 0.75) edge[dashed] (int);
  \node[ext, minimum size=0.2cm] (ext1) at (2,0.25) {$\scriptstyle j$};
  \path[->] (ext1) edge[dashed] (1.2,0.25);
  \node[ext, minimum size=0.2cm] (ext2) at (2,-0.25) {$\scriptstyle j$};
  \path[->] (1.2,-0.25) edge[dashed] (ext2);
  \end{tikzpicture}
\]
In total each graph has $M$ external vertices with an outgoing dashed edge and $N$ external vertices with an ingoing dashed edge:
\[
M\times\left(
	\begin{tikzpicture}
		\node[ext, minimum size=0.2cm] (e2) at (0,0) {};
		\path[->] (e2) edge[dashed] (-0.75,0);
  	\end{tikzpicture}
  	\right)
  	\quad\quad\quad
  	N\times\left(
  	\begin{tikzpicture}
		\node[ext, minimum size=0.2cm] (e2) at (0,0) {};
		\path[->] (-0.75,0) edge[dashed] (e2);
  	\end{tikzpicture}
  	\right)
  	\]
The following picture shows a graph of the auxiliary graph complex $G^\infty_{2,3}$:
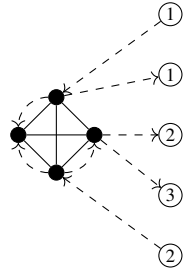
\[
\begin{tikzpicture}[baseline=-.65ex]
  \node[int] (v1) at (0, 0.5) {};
  \node[int] (v2) at (0.5, 0) {};
  \node[int] (v3) at (0, -0.5) {};
  \node[int] (v4) at (-0.5, 0) {};
  \draw (v1) edge (v2) edge (v3) edge (v4) (v2) edge (v3) edge (v4) (v3) edge (v4);
  \node[ext, minimum size=0.2cm] (e1) at (1.5, 1.6) {$\scriptstyle 1$};
  \node[ext, minimum size=0.2cm] (e2) at (1.5, 0.8) {$\scriptstyle 1$};
  \node[ext, minimum size=0.2cm] (e3) at (1.5, 0) {$\scriptstyle 2$};
  \node[ext, minimum size=0.2cm] (e4) at (1.5, -0.8) {$\scriptstyle 3$};
   \node[ext, minimum size=0.2cm] (e5) at (1.5, -1.6) {$\scriptstyle 2$};
  \path[->] (e1) edge[dashed] (v1);
  \path[->] (v1) edge[dashed] (e2);
  \path[->] (v2) edge[dashed] (e3);
  \path[->] (v2) edge[dashed] (e4);
  \path[->] (e5) edge[dashed] (v3);
  \draw[->](v1) [out=180, in=90] edge[dashed] (v4);
  \draw[->](v3) [out=0, in=-90] edge[dashed] (v2);
  \draw[->](v3) [out=180, in=-90] edge[dashed] (v4);
\end{tikzpicture}
\]
Note that the subgraph consisting of the internal vertices and the solid undirected edges must be connected.

The natural map $G_{M,N}^\infty\to G_{M,N}^g$ has the following combinatorial description. 
Let $\Gamma\in G_{M,N}^\infty$ be a graph as above. It has edges of two types (solid or dashed), but no decorations in $V_g$ or $V_g^*$. 
Then we replace each dashed edge (say $(u,v)$, from vertex $u$ to vertex $v$) by one copy of the $\frac 12$-diagonal element 
\[
\Delta^{\frac 12} =(-1)^m
\sum_{i=1}^g \alpha_i \otimes \beta_i
\]
where $V_g=span\{\alpha_1,\dots,\alpha_g\}$ and $V_g^*=span\{\beta_1,\dots,\beta_g\}$.
More precisely, the first factor $\alpha_i$ will be multiplied into the decoration of vertex $u$ and the second factor $\beta_i$ is multiplied into the decoration at vertex $v$.
Proceeding in this manner for all dashed edges, we obtain a linear combinations of graphs with no dashed edges but decorations in $V_g\oplus V_g^*$.
This linear combination is automatically $\GL_g$-invariant and hence defines an element of $G_{M,N}^g$, which is the image of the graph $\Gamma$.

The differential on $G_{M,N}^\infty$ has two components
\[
	d = d_{contract} + d_{cut}.
\]
$d_{contract}$ acts on all solid edges by contracting the edge 
\[
	d_{contract}\Gamma= \sum_{e} \Gamma/e.
\]
If the two internal vertices incident to the contracted edge $e$ are at position one and two in the ordering of the vertices, then the sign is positive and the differential acts on the decorations of the two vertices by multiplying them using the commutative product of the symmetric algebra $S\left(V_g\oplus V_g^*\right)$. In the graphical representation this means that all solid and dashed edges incident to the two vertices are finally connected to the remaining vertex. 

The second part of the differential, $d_{cut}$, acts also on all solid edges, cuts them and replaces them by the diagonal element $\Delta \in S^2(V_g\oplus V_g^*)$

\[
 	\Delta = (-1)^m \sum_{i=1}^g (\alpha_i\otimes \beta_i - \beta_i\otimes \alpha_i).
\]
In the graphical representation this corresponds to connecting the two internal vertices incident to the cut edge with dashed edges in either direction. This part of the differential vanishes if the subgraph consisting of the inner vertices and the solid undirected edges of the resulting graph is not connected anymore. 

Using the graphical notation the total differential of this graph complex is given as follows. We sum over all solid undirected edges $e$ and replace it by three terms. First, we contract the edge $e$ and connect all incident edges of the two affected vertices to the one remaining vertex. Second, we replace the solid undirected edge $e$ by directed dashed edges in ether direction, whereat the latter two terms have opposite sign. If the subgraph consisting of the internal vertices and the solid undirected edges of the resulting graph is not connected any more, then the second therm of the differential vanishes.
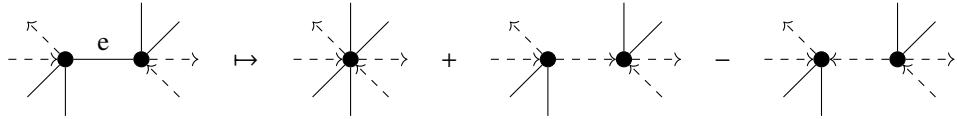
\[
	\begin{tikzpicture}
  		\node[int] (v1) at (0,0) {};
		\node[int] (v2) at (1,0) {};
		\draw (v1) edge node[midway, above]{e } (v2);
		\path[->] (-0.75,0) edge[dashed] (v1);
		\path[->] (v1) edge[dashed] (-0.5,0.5);
		\draw (-0.5,-0.5) edge (v1);
		\draw (0,-0.75) edge (v1);
		\path[->] (1.5,-0.5) edge[dashed] (v2);
		\path[->] (v2) edge[dashed] (1.75,0);
		\draw (1, 0.75) edge (v2);
		\draw (1.5, 0.5) edge (v2);
  	\end{tikzpicture}
  	\quad\mapsto\quad
  	\begin{tikzpicture}
  		\node[int] (v1) at (0,0) {};
		\path[->] (-0.75,0) edge[dashed] (v1);
		\path[->] (v1) edge[dashed] (-0.5,0.5);
		\draw (-0.5,-0.5) edge (v1);
		\draw (0,-0.75) edge (v1);
		\path[->] (0.5,-0.5) edge[dashed] (v1);
		\path[->] (v1) edge[dashed] (0.75,0);
		\draw (0, 0.75) edge (v1);
		\draw (0.5, 0.5) edge (v1);
  	\end{tikzpicture}
  	\quad+\quad
	\begin{tikzpicture}
  		\node[int] (v1) at (0,0) {};
		\node[int] (v2) at (1,0) {};
		\draw[->] (v1) edge[dashed] (v2);
		\path[->] (-0.75,0) edge[dashed] (v1);
		\path[->] (v1) edge[dashed] (-0.5,0.5);
		\draw (-0.5,-0.5) edge (v1);
		\draw (0,-0.75) edge (v1);
		\path[->] (1.5,-0.5) edge[dashed] (v2);
		\path[->] (v2) edge[dashed] (1.75,0);
		\draw (1, 0.75) edge (v2);
		\draw (1.5, 0.5) edge (v2);
  	\end{tikzpicture}
  	\quad-\quad
  	\begin{tikzpicture}
  		\node[int] (v1) at (0,0) {};
		\node[int] (v2) at (1,0) {};
		\draw[->] (v2) edge[dashed] (v1);
		\path[->] (-0.75,0) edge[dashed] (v1);
		\path[->] (v1) edge[dashed] (-0.5,0.5);
		\draw (-0.5,-0.5) edge (v1);
		\draw (0,-0.75) edge (v1);
		\path[->] (1.5,-0.5) edge[dashed] (v2);
		\path[->] (v2) edge[dashed] (1.75,0);
		\draw (1, 0.75) edge (v2);
		\draw (1.5, 0.5) edge (v2);
  	\end{tikzpicture}
\]

\subsection{Transformed auxiliary graph complex $\tilde{G}_{M,N}^\infty$}
\label{sec:auxiliary_G}
Let us now define a new transformed auxiliary graph complex $\tilde{G}_{M,N}^\infty$ which is isomorphic to the previous graph complex $G_{M,N}^\infty$ by applying the following coordinate transformation to dashed edges incident to internal vertices only.
\[
	\begin{tikzpicture}
  		\node[int] (v1) at (0,0) {};
		\node[int] (v2) at (1,0) {};
		\draw (v1) edge[dashed] node[midway, above]{$\oplus$} (v2);
  	\end{tikzpicture}
  	~:=~
  	\begin{tikzpicture}
  		\node[int] (v1) at (0,0) {};
		\node[int] (v2) at (1,0) {};
		\draw[->] (v1) edge[dashed] (v2);
  	\end{tikzpicture}
  	~+~
  	\begin{tikzpicture}
  		\node[int] (v1) at (0,0) {};
		\node[int] (v2) at (1,0) {};
		\draw[->] (v2) edge[dashed] (v1);
  	\end{tikzpicture} 	
\]
\[
	\begin{tikzpicture}
  		\node[int] (v1) at (0,0) {};
		\node[int] (v2) at (1,0) {};
		\draw (v1) edge[dashed] node[midway, above]{$\ominus$} (v2);
  	\end{tikzpicture}
  	~:=~
  	\begin{tikzpicture}
  		\node[int] (v1) at (0,0) {};
		\node[int] (v2) at (1,0) {};
		\draw[->] (v1) edge[dashed] (v2);
  	\end{tikzpicture}
  	~-~
  	\begin{tikzpicture}
  		\node[int] (v1) at (0,0) {};
		\node[int] (v2) at (1,0) {};
		\draw[->] (v2) edge[dashed] (v1);
  	\end{tikzpicture}
\]

In order to define the differential on the new transformed graph complex we sum again over all solid edges $e$, contract them and replace them by a $\ominus$-edge.
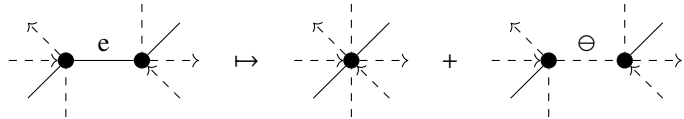
\[
	\begin{tikzpicture}
  		\node[int] (v1) at (0,0) {};
		\node[int] (v2) at (1,0) {};
		\draw (v1) edge node[midway, above]{e} (v2);
		\path[->] (-0.75,0) edge[dashed] (v1);
		\path[->] (v1) edge[dashed] (-0.5,0.5);
		\draw (-0.5,-0.5) edge (v1);
		\draw (0,-0.75) edge[dashed] (v1);
		\path[->] (1.5,-0.5) edge[dashed] (v2);
		\path[->] (v2) edge[dashed] (1.75,0);
		\draw (1, 0.75) edge[dashed] (v2);
		\draw (1.5, 0.5) edge (v2);
  	\end{tikzpicture}
  	\quad \mapsto\quad
  	\begin{tikzpicture}
  		\node[int] (v1) at (0,0) {};
		\path[->] (-0.75,0) edge[dashed] (v1);
		\path[->] (v1) edge[dashed] (-0.5,0.5);
		\draw (-0.5,-0.5) edge (v1);
		\draw (0,-0.75) edge[dashed] (v1);
		\path[->] (0.5,-0.5) edge[dashed] (v1);
		\path[->] (v1) edge[dashed] (0.75,0);
		\draw (0, 0.75) edge[dashed] (v1);
		\draw (0.5, 0.5) edge (v1);
  	\end{tikzpicture}
  	\quad+\quad
	\begin{tikzpicture}
  		\node[int] (v1) at (0,0) {};
		\node[int] (v2) at (1,0) {};
		\draw (v1) edge[dashed] node[midway, above]{$\ominus$} (v2);
		\path[->] (-0.75,0) edge[dashed] (v1);
		\path[->] (v1) edge[dashed] (-0.5,0.5);
		\draw (-0.5,-0.5) edge (v1);
		\draw (0,-0.75) edge[dashed] (v1);
		\path[->] (1.5,-0.5) edge[dashed] (v2);
		\path[->] (v2) edge[dashed] (1.75,0);
		\draw (1, 0.75) edge[dashed] (v2);
		\draw (1.5, 0.5) edge (v2);
  	\end{tikzpicture}
\]

\subsection{Cohomology of the auxiliary graph complex $gr^{W,H}\tilde{G}_{M,N}^\infty$}
 Let us consider the bounded filtration of $gr^{W,H}\tilde{G}_{M,N}^\infty$ with respect to the number of internal vertices $v$ and consider the associated graded complex $gr^vgr^{W,H}\tilde{G}_{M,N}^\infty$. The differential $d^v$ on this graded complex acts by summing over all solid edges and replacing them with $\ominus$-edges.

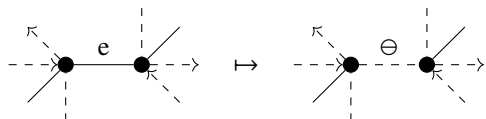
\[
	\begin{tikzpicture}
  		\node[int] (v1) at (0,0) {};
		\node[int] (v2) at (1,0) {};
		\draw (v1) edge node[midway, above]{e} (v2);
		\path[->] (-0.75,0) edge[dashed] (v1);
		\path[->] (v1) edge[dashed] (-0.5,0.5);
		\draw (-0.5,-0.5) edge (v1);
		\draw (0,-0.75) edge[dashed] (v1);
		\path[->] (1.5,-0.5) edge[dashed] (v2);
		\path[->] (v2) edge[dashed] (1.75,0);
		\draw (1, 0.75) edge[dashed] (v2);
		\draw (1.5, 0.5) edge (v2);
  	\end{tikzpicture}
  	\quad \mapsto\quad
	\begin{tikzpicture}
  		\node[int] (v1) at (0,0) {};
		\node[int] (v2) at (1,0) {};
		\draw (v1) edge[dashed] node[midway, above]{$\ominus$} (v2);
		\path[->] (-0.75,0) edge[dashed] (v1);
		\path[->] (v1) edge[dashed] (-0.5,0.5);
		\draw (-0.5,-0.5) edge (v1);
		\draw (0,-0.75) edge[dashed] (v1);
		\path[->] (1.5,-0.5) edge[dashed] (v2);
		\path[->] (v2) edge[dashed] (1.75,0);
		\draw (1, 0.75) edge[dashed] (v2);
		\draw (1.5, 0.5) edge (v2);
  	\end{tikzpicture}
\]
The differential vanishes if the subgraph consisting of the internal vertices and solid edges would be split in two parts. 

\begin{lemma}
	The cohomology of the associated graded complex $\left(gr^vgr^{W,H}\tilde{G}_{M,N}^\infty,d^v\right)$ consists of graphs which are solid trees, i.e. the subgraph consisting of the internal vertices and the solid edges forms a tree. 
\end{lemma}
\begin{proof}
	The considered complex $\left(gr^vgr^{W,H}\tilde{G}_{M,N}^\infty,d^v\right)$ corresponds 
	to the graph complex with two coloured edges as described before in section \ref{sec:CGamma}. Hence, its cohomology consists of graphs which are solid trees by Lemma \ref{lem:CGamma}. 
\end{proof}

Therefore, we get a lower bound for the degree $k$ part of the comohology of the graded complex 
\[
H^k\left(gr^vgr^{W,H}\tilde{G}_{M,N}^\infty,d^v\right).
\]

\begin{lemma}
		$H^k\left(gr^vgr^{W,H}\tilde{G}_{M,N}^\infty,d^v\right)=0~,~\text{for}~k< - k_{crit}^L(W,H)$.
\end{lemma}
\begin{proof}
Let $a$ denote the number of decorations with an element of $H^m(W_{g,1})$, and $b$ the number of decorations with an element of $H^{m+1}(W_{g,1})$ of the related Graph in the original graph complex $\G_{(g),1}$. Furthermore, let $v$ denote the number of internal vertices and $e$ the number of solid edges. Because the internal vertices are at least trivalent we can deduce an upper bound for $v$
	\[
		3v\leq a+b+2e.
	\]
	The imbalance is given as $H=a-b$ and the weight as $W=2(e-v)+a+b$. Therefore, the upper bound for $v$ is equivalent to
	\[ 
		v\leq a+b+2(e-v)=W.
	\]
	The degree is given by 
	\[
		k=-(nv-(n-1)e-am-b(m+1)+1)=-nv+(n-1)e+am+b(m+1)-1.
	\]
	Furthermore, since the solid subgraph forms a tree we have $e=v-1$ and therefore the degree simplifies to
	\[
		k=-v-n+am+b(m+1)
	\]
	and the weight to
	\[
		W=a+b-2.
	\]
	We express $a$ and $b$ in terms of $W$ and $H$
	\[
	\begin{aligned}
		a&=\frac{1}{2}(W+H)+1\\
		b&=\frac{1}{2}(W-H)+1.
	\end{aligned}
	\]
	Together with the upper bound for the number of vertices $v$ this leads to a lower bound for the degree $k$ (note that $n=2m+1$)
	\[
		k\geq -W + \frac{1}{2}m(W+H)+\frac{1}{2}(m+1)(W-H)= \frac{W}{2}(2m-1)-\frac{H}{2}=-k_{crit}^L(W,H).
	\]	
	\end{proof}
	
	We can then consider the spectral sequence starting with the associated graded complex induced by the filtration with respect to the number of vertices $v$. Since the cohomology on the first page vanishes for $k <-k^L_{crit}$, the spectral sequence converges towards the cohomology of the initial complex $gr^{W,H}\tilde{G}_{M,N}^\infty$ which is therefore acyclic for $k <-k^L_{crit}$. We can therefore deduce the following lemma:
	
\begin{lemma}
\label{lem:G infty acyclic}
	$H^k\left(gr^{W,H}G_{M,N}^\infty\right)=0~,~\text{for}~k< - k_{crit}^L(W,H)$.
\end{lemma}

\subsection{Proof of part (ii) of Theorem \ref{thm:main cohom GC vanishing}}
We like to show that for all $W\geq 1$, $g\geq W+2$ and $H\in \Z$   
\[
    \gr^{(W,H)} H^{k}\left(\GC_{(g),1}\right) = 0 ~,~\text{for}~ k>k_{crit}^L(W,H).
\] 
Since the bigrading is compatible with the differential this corresponds to 
\[
	H^k\left(gr^{(W,H)}\GC_{(g),1}\right)=0~,~\text{for}~ k>k_{crit}^L(W,H).
\]
and to the corresponding expression for its graded dual complex
\[
	H^k\left(gr^{(W,H)}\G_{(g),1}\right)=0~,~\text{for}~ k<-k_{crit}^L(W,H).
\]
By Lemma \ref{lem:G_vanishing} it is sufficient to check that 
\[
H^k\left(\gr^{(W,H)} G^\infty_{M,N} \right)=0
\] 
for $k< -k^L_{crit}(W,H)$ and for any $M,N$ such that $M+N\leq W+2$.
The latter follows from Lemma \ref{lem:G infty acyclic}, which proves part (ii) of Theorem \ref{thm:main cohom GC vanishing}.

%
%
\section{Chevalley-Eilenberg complex of $\GC_{(g),1}$}

\subsection{Chevalley-Eilenberg complex}
 We consider the Chevalley-Eilenberg complex 
 $$
 C_{CE}\left(\GC_{(g),1}\right)
 :=
 \Bar^c\left(\GC_{(g),1}^c\right)
 =\Bar^c\left(\G_{(g),1}\right)=\fG_{(g),1} \oplus \Q,
 $$
 defined as the cobar construction of the graded dual of the dg Lie algebra $\GC_{(g),1}$, or equivalently via the graph complex $\fG_{(g),1}$ introduced in section \ref{sec:GCs}.
As a graded vector space it is isomorphic to the symmetric algebra
\[
	S\left(\left(\GC_{(g),1}\right)^c[-1]\right) = S\left(\G_{(g),1}[-1] \right).
\]
Again, let $V_g\cong H^m(W_{g,1})$ and $V_g^*\cong H^{m+1}(W_{g,1})$ correspond to the defining and adjoint representations of $\GL_g$ respectively. Every vertex of a graph in $\G_{(g),1}$ is decorated by elements of $S\left(\bar{H}(W_{g,1})\right)\cong S\left(V_g\oplus V_g^*\right)$.

We consider the weight $W$ and imbalance $H$ part of the comohology of the Chevalley-Eilenberg complex $gr^{(W,H)}H\left(C_{CE}\left(\GC_{(g),1}\right)\right)$. Since the bigrading is compatible with the differential this is isomorphic to \\$H\left(gr^{(W,H)}C_{CE}\left(\GC_{(g),1}\right)\right)$. 
 
\subsection{Stabilisation for large $g$}
\label{sec:stabil_F}

In order to apply the invariant theory for $\GL_g$ we study the auxiliary graph complex

\[
F^g_{M,N} :=\left(C_{CE}\left(\GC_{(g),1}\right)\otimes V_g^M\otimes \left(V_g^*\right)^N\right)^{\GL_g}
\]
and its graded version
\[
gr^{(W,H)} F^g_{M,N} =gr^{(W,H)}\left(C_{CE}\left(\GC_{(g),1}\right)\otimes V_g^M\otimes \left(V_g^*\right)^N\right)^{\GL_g}=\left(gr^{(W,H)}C_{CE}\left(\GC_{(g),1}\right)\otimes V_g^M\otimes \left(V_g^*\right)^N\right)^{\GL_g}.
\]

Like in the previous section \ref{sec:stabil_G} we have sequence of natural projection maps 
\begin{equation}
\label{equ:Fgtower}
	\cdots \to F^{g+1}_{M,N} \to F^{g}_{M,N}\to F^{g-1}_{M,N} \to \cdots\, .
\end{equation}

\begin{lemdef}
\label{lemdef:Finfty}
	The sequence above stabilizes in each finite weight.
	That is, for each $M,N,W,H$ and $2g\geq 3W+M+N$ the map 
	\[
		gr^{(W,H)}F_{M,N}^{g+1} \to gr^{(W,H)}F_{M,N}^{g}
	\]
	is an isomorphism.
	We define the bigraded dg vector space 
	\[
		F_{M,N}^{\infty} = \bigoplus_{W,H} 	gr^{(W,H)}F_{M,N}^{\infty}
	\]
	such that 
	\[
		gr^{(W,H)}F_{M,N}^{\infty} := gr^{(W,H)}F_{M,N}^{g}
	\]
	for any $g\geq \frac12( 3W+M+N)$. In other words, $F_{M,N}^{\infty}$ is the weight-wise limit of the tower \eqref{equ:Fgtower}.
\end{lemdef}
Note that we in particular obtain natural maps 
\[
	F_{M,N}^{\infty} \to F_{M,N}^{g}
\]
for every $g$.

\begin{proof}
The proof is identical to the one for Lemma \ref{lemdef:Ginfty} with the difference that we consider the Chevalley-Eilenberg complex $C_{CE}\left(\GC_{(g),1}\right)$ and hence not only have connected graphs in $\G_{(g),1}$ but also graphs with several connected components in $\fG_{(g),1}$.

The same argument based on fundamental theorems of invariant theory for $\GL_g$ as used in section \ref{sec:stabil_G} leads to a stabilisation of $gr^{(W,H)}F_{M,N}^{g}$ for large enough $g$, i.e. it becomes independent of the genus $g$ and if $A\neq B$ it vanishes.

The condition $A=B$ leads to $H=a-b=N-M$. In order to derive a lower bound for the genus to ensure stabilisation of $gr^{(W,H)}F_{M,N}^{g}$, we use the fact that the vertices of the considered graphs are at least trivalent (including the decorations), $3v\leq 2e+D$. Hence, we get $3(v-e) \leq -e + D \leq D$, where we also used that $e\geq 0$. This leads to an upper bound for the weight $W=2(e-v) + a + b \geq -2/3 D + D$. Finally we reach an upper bound for the number of decorations, $D \leq 3W$.

If $2g\geq A+B=a+M+b+N=D+M+N$ the considered complex stabilises. Using the derived upper bound for $D$, this is fulfilled if $2g\geq 3W+M+N$.
\end{proof}

The complexes $F_{M,N}^{g}$ have the following significance:
\begin{lemma}\label{lem:F and CE vanishing}
If $H^k\left(\gr^{(W,H)}F_{M,N}^{g}\right)=0$ for fixed $g,W,H,k$ and all $M,N$ such that $M+N\leq 3W$ then 
\begin{equation}\label{equ:lem CE_GC_vanishing}
H^k\left(\gr^{(W,H)}C_{CE}(\GC_{(g),1})\right)=0.
\end{equation}
In particular, if $H^k\left(\gr^{(W,H)}F_{M,N}^{\infty}\right)=0$ for all $M,N$ such that $M+N\leq 3W$, then \eqref{equ:lem CE_GC_vanishing} holds for all $g$ such that $g\geq 3W$.
\end{lemma}
\begin{proof}
 By assumption we know that 
 \begin{align*}
 0&=H^k\left(\gr^{(W,H)}F_{M,N}^{g}\right)\\
 &=H^k\left(\left(gr^{(W,H)}C_{CE}\left(\GC_{(g),1}\right)\otimes V_g^M\otimes \left(V_g^*\right)^N\right)^{\GL_g}\right)\\
 &=\left(H^k\left(gr^{(W,H)}C_{CE}\left(\GC_{(g),1}\right)\right)\otimes V_g^M\otimes \left(V_g^*\right)^N\right)^{\GL_g}
 \end{align*}
 for all $M,N$ with $M+N\leq 3W$ and $g\geq 3W$. In the last equation we used the fact that we can interchange taking the cohomology with taking the $\GL_g$-invariants. 
 
  Due to the upper bound for the number of decorations $D\leq 3W$ we know that $H^k\left(\gr^{(W,H)}C_{CE}(\GC_{(g),1})\right)$ is a $\GL_g$-representation of order $3W$. Therefore, the claim follows from Lemma \ref{lem:rep_gl_g}. 
  
  The second statement of the Lemma is again obtained by combining the first statement with the stabilization result Lemma/Definition \ref{lemdef:Finfty}.

 \end{proof}

\subsection{Combinatorial form of $F_{M,N}^\infty$}
\label{sec:combinatorial_F}

The auxiliary graph complex $F_{M,N}^\infty$ defined above has a combinatorial description as a graph complex, as can be seen from the proof of Lemma/Definition \ref{lemdef:Finfty}. The combinatorial definition of the graph complex $F_{M,N}^\infty$ is identical to the one of $G_{M,N}^\infty$ as described in section \ref{sec:combinatorial_G} with the one difference that  the subgraph consisting of the internal vertices and the solid edges does not need to be connected.


The following picture shows a graph of the auxiliary graph complex $F^\infty_{2,3}$.
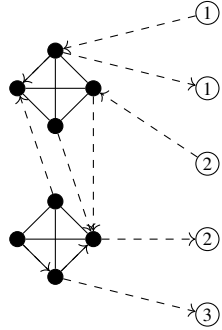
\[
\begin{tikzpicture}[baseline=-.65ex]
  \node[int] (v1) at (0, 1.5) {};
  \node[int] (v2) at (0.5, 1) {};
  \node[int] (v3) at (0, 0.5) {};
  \node[int] (v4) at (-0.5, 1) {};
  \draw (v1) edge (v2) edge (v3) edge (v4) (v2) edge (v3) edge (v4) (v3) edge (v4);
  \node[int] (v5) at (0, -0.5) {};
  \node[int] (v6) at (0.5, -1) {};
  \node[int] (v7) at (0, -1.5) {};
  \node[int] (v8) at (-0.5, -1) {};
  \draw (v5) edge (v6) edge (v7) edge (v8) (v6) edge (v7) edge (v8) (v7) edge (v8);
  \node[ext, minimum size=0.2cm] (e1) at (2, 2) {$\scriptstyle 1$};
  \node[ext, minimum size=0.2cm] (e2) at (2, 1) {$\scriptstyle 1$};
  \node[ext, minimum size=0.2cm] (e3) at (2, 0) {$\scriptstyle 2$};
  \node[ext, minimum size=0.2cm] (e4) at (2, -1) {$\scriptstyle 2$};
\node[ext, minimum size=0.2cm] (e5) at (2, -2) {$\scriptstyle 3$};
  \draw[->] (v1) edge[dashed] (e2);
  \draw[->] (e3) edge[dashed] (v2);
  \draw[->] (e1)  edge[dashed] (v1);
  \draw[->] (v6)  edge[dashed] (e4);
  \draw[->] (v7)  edge[dashed] (e5);
  \draw[->](v5) edge[dashed] (v4);
  \draw[->](v2)  edge[dashed] (v6);
  \draw[->](v3) edge[dashed] (v6);
  \draw[->](v7)  edge[dashed] (v6);
  \draw[->](v8)  edge[dashed] (v7);
  \draw[->](v1)  edge[dashed] (v4);
\end{tikzpicture}
\]

The natural map $F_{M,N}^\infty\to F_{M,N}^g$ is analogous to the map $G_{M,N}^\infty\to G_{M,N}^g$.

The differential on $F_{M,N}^\infty$ has also two components
\[
	d = d_{contract} + d_{cut}
\]
which are defined identically to the differential on $G_{M,N}^\infty$ defined in section \ref{sec:combinatorial_G} with the difference that $d_{cut}$ does not vanish if it cuts the subgraph consisting of the internal vertices and solid edges in two components.

\subsection{Transformed auxiliary graph complex $\tilde{F}_{M,N}^\infty$}
\label{sec:auxiliary_F}
We define again a new transformed auxiliary graph complex $\tilde{F}_{M,N}^\infty$ in the same way as  the auxiliary graph complex $\tilde{G}_{M,N}^\infty$ in section \ref{sec:auxiliary_G}. Again the differential does not vanish if it cuts the subgraph consisting of the internal vertices and solid edges in two components.

\subsection{Cohomology of the auxiliary graph complex $\tilde{F}_{M,N}^\infty$ and proof of Theorem \ref{thm:main CE all}}

We define $P_{M,N}\subset \tilde{F}_{M,N}^\infty$ as the subspace spanned by graphs without solid edges and "$\ominus$" dashed edges. The following pictures shows an element of $P_{1,1}$ and $P_{3,4}$ respectively:

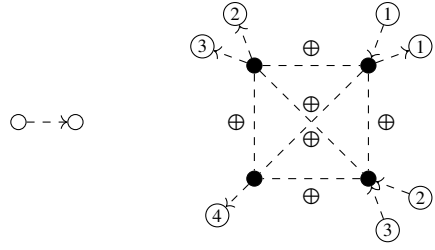
\[
	\begin{tikzpicture}
		\node[ext, minimum size=0.2cm] (e1) at (-0.75,0) {};
		\node[ext, minimum size=0.2cm] (e2) at (0,0) {};
		\path[->] (e1) edge[dashed] (e2);
  	\end{tikzpicture}
  	\quad\quad\quad\quad
	\begin{tikzpicture}
		\node[int] (v1) at (-.75, .75) {};
  		\node[int] (v2) at (.75, .75) {};
  		\node[int] (v3) at (.75, -.75) {};
  		\node[int] (v4) at (-.75, -.75) {};
  		\draw (v1) edge[dashed] node[midway, above]{$\oplus$} (v2) edge[dashed] node[midway, above]{$\oplus$} (v3) edge[dashed] node[midway, left]{$\oplus$} (v4) (v2) edge[dashed] node[midway, right]{$\oplus$} (v3) edge[dashed] node[midway, below]{$\oplus$} (v4) (v3) edge[dashed] node[midway, below]{$\oplus$} (v4);
  		\node[ext, minimum size=0.2cm] (e1) at (35:1.75) {$\scriptstyle 1$};
  		\node[ext, minimum size=0.2cm] (e2) at (55:1.75) {$\scriptstyle 1$};
		\node[ext, minimum size=0.2cm] (e3) at (125:1.75) {$\scriptstyle 2$};
		\node[ext, minimum size=0.2cm] (e4) at (145:1.75) {$\scriptstyle 3$};
  		\node[ext, minimum size=0.2cm] (e5) at (-35:1.75) {$\scriptstyle 2$};
  		\node[ext, minimum size=0.2cm] (e6) at (-55:1.75) {$\scriptstyle 3$};
  		\node[ext, minimum size=0.2cm] (e7) at (-135:1.75) {$\scriptstyle 4$};
  		\path[->] (v2) edge[dashed] (e1);
  		\path[->] (e2) edge[dashed] (v2);
		\path[->] (v1) edge[dashed] (e3);
  		\path[->] (v1) edge[dashed] (e4);
  		\path[->] (e5) edge[dashed] (v3);
  		\path[->] (e6) edge[dashed] (v3);
  		\path[->] (v4) edge[dashed] (e7);
  	\end{tikzpicture}
\]

\begin{lemma}
\label{lem:quasi-iso}
	The inclusion $P_{M,N} \hookrightarrow \tilde{F}_{M,N}^\infty \cong F_{M,N}^\infty$ is a quasi-isomorphism.  Hence, the cohomology of the auxiliary graph complex is isomorphic to $P_{M,N}$.
\end{lemma}

\begin{proof}
In a first step we apply Lemma \ref{lemma:inclusion} with $V=\tilde{F}_{M,N}^\infty$ and $U=P_{M,N}$.

Therefore, we are left to show that $\tilde{F}_{M,N}^\infty/P_{M,N}$ is acyclic. Let us consider the bounded filtration of $\tilde{F}_{M,N}^\infty/P_{M,N}$ with respect to the number of internal vertices $v$ and the associated graded complex $gr^v\tilde{F}_{M,N}^\infty/P_{M,N}$. From Lemma \ref{lemma:graded} it can be deduced that if the associate graded complex $gr^v\tilde{F}_{M,N}^\infty/P_{M,N}$ is acyclic then so is $\tilde{F}_{M,N}^\infty/P_{M,N}$. 

Finally, we have to show that the associated graded complex $gr^v\tilde{F}_{M,N}^\infty/P_{M,N}$ is acyclic. The differential $d^v$ of this graded complex reads as follows. We sum over all solid edges $e$ and replace it by a $\ominus$-edge.
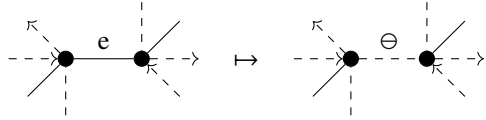
\[
	\begin{tikzpicture}
  		\node[int] (v1) at (0,0) {};
		\node[int] (v2) at (1,0) {};
		\draw (v1) edge node[midway, above]{e} (v2);
		\path[->] (-0.75,0) edge[dashed] (v1);
		\path[->] (v1) edge[dashed] (-0.5,0.5);
		\draw (-0.5,-0.5) edge (v1);
		\draw (0,-0.75) edge[dashed] (v1);
		\path[->] (1.5,-0.5) edge[dashed] (v2);
		\path[->] (v2) edge[dashed] (1.75,0);
		\draw (1, 0.75) edge[dashed] (v2);
		\draw (1.5, 0.5) edge (v2);
  	\end{tikzpicture}
  	\quad \mapsto\quad
	\begin{tikzpicture}
  		\node[int] (v1) at (0,0) {};
		\node[int] (v2) at (1,0) {};
		\draw (v1) edge[dashed] node[midway, above]{$\ominus$} (v2);
		\path[->] (-0.75,0) edge[dashed] (v1);
		\path[->] (v1) edge[dashed] (-0.5,0.5);
		\draw (-0.5,-0.5) edge (v1);
		\draw (0,-0.75) edge[dashed] (v1);
		\path[->] (1.5,-0.5) edge[dashed] (v2);
		\path[->] (v2) edge[dashed] (1.75,0);
		\draw (1, 0.75) edge[dashed] (v2);
		\draw (1.5, 0.5) edge (v2);
  	\end{tikzpicture}\]
We can now define a homotopy $h^v$ corresponding to the inverse of this differential.
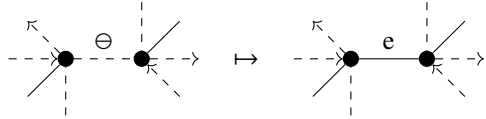
\[
	\begin{tikzpicture}
  		\node[int] (v1) at (0,0) {};
		\node[int] (v2) at (1,0) {};
		\draw (v1) edge[dashed] node[midway, above]{$\ominus$} (v2);
		\path[->] (-0.75,0) edge[dashed] (v1);
		\path[->] (v1) edge[dashed] (-0.5,0.5);
		\draw (-0.5,-0.5) edge (v1);
		\draw (0,-0.75) edge[dashed] (v1);
		\path[->] (1.5,-0.5) edge[dashed] (v2);
		\path[->] (v2) edge[dashed] (1.75,0);
		\draw (1, 0.75) edge[dashed] (v2);
		\draw (1.5, 0.5) edge (v2);
  	\end{tikzpicture}
  	\quad \mapsto\quad
	\begin{tikzpicture}
  		\node[int] (v1) at (0,0) {};
		\node[int] (v2) at (1,0) {};
		\draw (v1) edge node[midway, above]{e} (v2);
		\path[->] (-0.75,0) edge[dashed] (v1);
		\path[->] (v1) edge[dashed] (-0.5,0.5);
		\draw (-0.5,-0.5) edge (v1);
		\draw (0,-0.75) edge[dashed] (v1);
		\path[->] (1.5,-0.5) edge[dashed] (v2);
		\path[->] (v2) edge[dashed] (1.75,0);
		\draw (1, 0.75) edge[dashed] (v2);
		\draw (1.5, 0.5) edge (v2);
  	\end{tikzpicture}
\]

Finally we apply Lemma \ref{lem:homotopy} with $V=gr^v\tilde{F}_{M,N}^\infty/P_{M,N}$ and $A=d^vh^v+h^vd^v=\mathds{1}$. Therefore the associate graded complex $gr^v\tilde{F}_{M,N}^\infty/P_{M,N}$ is acyclic, which finishes the proof.
\end{proof}

Considering the graphs which are left in the subspace $P_{M,N}$ we can show that the cohomology of the weight $W$ and imbalance $H$ component of the auxiliary graph complex is concentrated in the critical degree.
\begin{lemma}
\label{lem:Finfty vanishing}
	\[
		H^k\left(gr^{(W,H)}F_{M,N}^\infty\right)	=0~,~\text{for}~k\neq k_{crit}^C(W,H).
	\]
\end{lemma}
\begin{proof}
	We will show that $\gr^{(W,H)}P_{M,N}$ is concentrated in degree $k_{crit}^C(W,H)$. Then the lemma follows from Lemma \ref{lem:quasi-iso} before.

	Let $v$ denote the number of internal vertices, $e$ the number of solid edges, $a$ the number of decorations with an element of $H^m(W_{g,1})$, and $b$ the number of decorations with an element of $H^{m+1}(W_{g,1})$ of the related Graph in the original graph complex $\fG_{(g),1}$. Since the remaining graphs in the subspace $P$ do not have any solid edges left, we have $e=0$ for all graphs in the cohomology. Therefore, the degree of graphs in the cohomology is given by
	\[
		k= -(2m+1)v + am +b(m+1),
	\]
	and the weight by
	\[
		W=a+b-2v.
	\]
	We express $a$ and $b$ in terms of the weight $W$, the number of vertices $v$, and the imbalance
	\[
		H=a-b:
	\]
	\[
	\begin{aligned}
		a&=\frac{1}{2}(W+H)+v\\
		b&=\frac{1}{2}(W-H)+v.
	\end{aligned}
	\]
	Therefore, the degree of graphs in the cohomology corresponds to the critical degree
	\[
		k=\frac{W}{2}(2m+1)-\frac{H}{2}= k_{crit}^C(W,H).
	\]
	
\end{proof}	

\subsection{Proof of Theorem \ref{thm:main CE all}}
Fix $W,H,g\geq 3W$. We want to show that 
\[
gr^{(W,H)}H^k\left(C_{CE}\left(\GC_{(g),1}\right)\right) = 0~,~\text{for}~ k \neq k_{crit}^C(W, H).
\]
Since the bigrading is compatible with the differential this corresponds to 
\[
H^k\left(gr^{(W,H)}C_{CE}\left(\GC_{(g),1}\right)\right) = 0,~\text{for}~ k \neq k_{crit}^C(W, H).
\]
By Lemma \ref{lem:F and CE vanishing} it is sufficient to check that
\[
H^k\left(gr^{(W,H)}F_{M,N}^{\infty}\right) = 0.
\]
for $k \neq k_{crit}^C(W, H)$ and for any $M$, $N$ such that $M+N\leq 3W$. The latter follows from Lemma \ref{lem:Finfty vanishing}. Hence, Theorem \ref{thm:main CE all} follows.
 
\subsection{Proof of Theorem \ref{thm:main Koszul}}

Having established the vanishing Theorems \ref{thm:main cohom GC vanishing} and \ref{thm:main CE all} for the Lie algebra and its Chevalley-Eilenberg cohomology, we obtain Theorem \ref{thm:main Koszul} as a direct consequence of Proposition \ref{prop:cohom}, with $W_0=\lfloor \frac g3 \rfloor$.

%


\subsection{A(nother) presentation of $H_{CE}(\GC_{(g),1})$}

Note that the differential on $C_{CE}\left(\GC_{(g),1}\right)$ is homogeneous of degree -1 with respect to the number of edges in graphs.
Hence the cohomology is also graded by the number of edges. Furthermore, the edge grading is compatible with the commutative product.

Let us denote by $A_{(g),1}\subset H_{CE}(\GC_{(g),1})$ the part of the cohomology of edge degree zero. It forms a graded subalgebra.
Since graphs cannot have negative numbers of edges, $A_{(g),1}$ is the part of $C_{CE}\left(\GC_{(g),1}\right)$ of edge degree zero, modulo the image of the part of edge degree 1 under the differential.
But a graph without edges is a union of single-vertex graphs 
\[
S_U=
\begin{tikzcd}
	\node[int,label={U}] (v) at (0,0) {};
\end{tikzcd},
\]
with $U\in S^{\geq 3}(\bar H(W_{g,1}))$.
The images under the differential of the one-edge graphs are unions of graphs of the form $S_U$ and 
\begin{equation}\label{equ:HCE rel}
S_{U'}S_{U''} +\sum_i \left(  \pm S_{\alpha_i U'}S_{\beta_i U''} -\pm  S_{\beta_i U'}S_{\alpha_i U''} \right)\, .
\end{equation}
That means that the graded commutative algebra $A_{(g),1}$ has a presentation with generators $S_U$, $U\in S^{\geq 3}(\bar H(W_{g,1}))$, and relations \eqref{equ:HCE rel} (plus the relation that $S_U$ is understood to be linear in $U$).
The stabilization result of the previous sections has the following variant for $A_{(g),1}$.

\begin{prop}
The inclusion $\gr^{W}A_{(g),1}\subset \gr^{W} H_{CE}(\GC_{(g),1})$ is an isomorphsism as soon as $g\geq 3W$. 
\end{prop}
\begin{proof}
We have to check that the cohomology $\gr^{W} H_{CE}(\GC_{(g),1})$ vanishes in edge degree $>0$. 
To this end, first note that Lemma \ref{lem:F and CE vanishing} continues to hold if we replace the cohomological degree $k$ therein by the edge degree.
Hence it suffices to show that the cohomology of $\gr^{W} F_{M,N}^\infty$ vanishes in edge degree $>0$. But that is shown in Lemma \ref{lem:quasi-iso}.
\end{proof}


\section{Chevalley-Eilenberg complex of $\GCex_{(g),1}$}

\subsection{Chevalley-Eilenberg complex}

 We consider the Chevalley-Eilenberg complex 
 $$
 C_{CE}\left(\GCex_{(g),1}\right)
 :=
 \Bar^c\left(\left(\GCex_{(g),1}\right)^c\right)
 = \Bar^c\left(\left(\osp_{g,1}^{nil}\ltimes \GC_{(g),1}\right)^c\right),
 $$
 defined as the cobar construction of the graded dual of the dg Lie algebra $\GCex_{(g),1}=\osp_{g,1}^{nil}\ltimes \GC_{(g),1}$.
As a graded vector space it is isomorphic to the symmetric algebra
\[
	S\left(\left(\GCex_{(g),1}\right)^c[-1]\right) \cong S\left(\left(\osp_{g,1}^{nil}\right)^*[-1]\right)\otimes S\left(\G_{(g),1}[-1]\right)\cong S\left(\left(\osp_{g,1}^{nil}\right)^*[-1]\right)\otimes \left(\fG_{(g),1} \oplus \Q \right).
\]
Again, let $V_g\cong H^m(W_{g,1})$ and $V_g^*\cong H^{m+1}(W_{g,1})$ correspond to the defining and adjoint representations of $\GL_g$ respectively. We can interpret $\osp_{g,1}^{nil} \in \Hom\left(V_g^*,V_g\right)\cong V_g\otimes V_g$ and $\left(\osp_{g,1}^{nil}\right)^* \in V_g^*\otimes V_g^*$. Furthermore, every vertex of a graph in $\fG_{(g),1}$ is decorated by elements of $S\left(\bar{H}(W_{g,1})\right)\cong S\left(V_g\oplus V_g^*\right)$. 
As pointed out in the introduction, we extend the $\Z\times\Z$-bigrading to $\GCex_{(g),1}$ by declaring the part $\osp_{g,1}^{nil}$ to be concentrated in weight $W=0$ and imbalance $H=-2$, compatible with the dg Lie structure.
In particular, $\osp_{g,1}^{nil}$ acts on $\GCex_{(g),1}$ respecting the weight grading and 
\[
H\left(\GCex_{(g),1}\right) = \osp_{g,1}^{nil}\ltimes H(\GC_{(g),1}).
\]

Again, we consider the weight $W$ and imbalance $H$ part of the cohomology of the Chevalley-Eilenberg complex $gr^{(W,H)}H\left(C_{CE}\left(\GCex_{(g),1}\right)\right)$. Since the bigrading is compatible with the differential this is isomorphic to \\$H\left(gr^{(W,H)}C_{CE}\left(\GCex_{(g),1}\right)\right)$. 

\subsection{Stabilisation for large $g$}

In order to apply the invariant theory for $\GL_g$ we study the auxiliary graph complex

\[
C^g_{M,N} :=\left(C_{CE}\left(\GCex_{(g),1}\right)\otimes V_g^M\otimes \left(V_g^*\right)^N\right)^{\GL_g}
\]
and its graded version
\[
gr^{(W,H)} C^g_{M,N} =gr^{(W,H)}\left(C_{CE}\left(\GCex_{(g),1}\right)\otimes V_g^M\otimes \left(V_g^*\right)^N\right)^{\GL_g}=\left(gr^{(W,H)}C_{CE}\left(\GCex_{(g),1}\right)\otimes V_g^M\otimes \left(V_g^*\right)^N\right)^{\GL_g}.
\]

We have again a sequence of natural projection maps 
\begin{equation}
\label{equ:Cgtower}
	\cdots \to C^{g+1}_{M,N} \to C^{g}_{M,N}\to C^{g-1}_{M,N} \to \cdots\, .
\end{equation}

\begin{lemdef}
\label{lemdef:Cinfty}
	The sequence above stabilizes in each finite weight.
	That is, for each $M,N,W,H$ and $2g \geq 6W-H+M+N$ the map 
	\[
		gr^{(W,H)}C_{M,N}^{g+1} \to gr^{(W,H)}C_{M,N}^{g}
	\]
	is an isomorphism.
	We define the bigraded dg vector space 
	\[
		C_{M,N}^{\infty} = \bigoplus_{W,H} 	gr^{(W,H)}C_{M,N}^{\infty}
	\]
	such that 
	\[
		gr^{(W,H)}C_{M,N}^{\infty} := gr^{(W,H)}C_{M,N}^{g}
	\]
	for any $g\geq \frac12(6W-H+M+N)$. In other words, $C_{M,N}^{\infty}$ is the weight-wise limit of the tower \eqref{equ:Cgtower}.
\end{lemdef}

Note that we in particular obtain natural maps 
\[
	C_{M,N}^{\infty} \to C_{M,N}^{g}
\]
for every $g$.

\begin{proof}
The proof is identical to the one of Lemma \ref{lemdef:Finfty} with the difference that we have additional contributions from  $S\left(\left(\osp_{g,1}^{nil}\right)^*[-1]\right)\cong S\left(\left(\left(V_g^*\otimes V_g^*\right)/S_{2}\right)[-1]\right)$.
The proof leads again to $\GL_g$ invariants of expressions of the form $V_g^{\otimes A} \otimes \left(V_g^*\right)^{\otimes B}$ with $A=a+M$ and $B=2d+b+N$ where $a$ denotes the number of decorations in $V_g$ and $b$ the number of decorations in $V_g^*$ of the graph and $2d$ the number of elements of $\osp_{g,1}^{nil}$. In this case the imbalance is defined as $H=a-b-2d$.

The same argument based on fundamental theorems of invariant theory for $\GL_g$ as used in section \ref{sec:stabil_G} leads to a stabilisation of $gr^{(W,H)}C_{M,N}^{g}$ for large enough $g$, i.e. it becomes independent of the genus $g$ and if $A\neq B$ it vanishes.

From the condition $A=B$ we deduce that $H=a-b-2d=N-M$. In order to derive a lower bound for the genus to ensure stabilisation of $gr^{(W,H)}C_{M,N}^{g}$, we use again the fact that the vertices of the considered graphs are at least trivalent (including the decorations), $3v\leq 2e+D$, and that $e\geq 0$ to reach the same upper bound for the number of decorations, $D\leq 3W$, as derived in section \ref{sec:stabil_F}. We know that $H+2d=a-b\leq a+b=D\leq 3W$. Therefore, we have $2d \leq 3W-H$ and $D+2d\leq 6W-H$.

If $2g\geq A+B=a+M+2d+b+N=D+2d+M+N$ the considered complex stabilises. Using the derived upper bound for $D+2d$, this is fulfilled if $2g\geq 6W-H+M+N$. 
\end{proof}

The above proof also shows the following Lemma.

\begin{lemma}\label{lem:HCEW}
The $\GL_g$ representation $\gr^{(W,H)}C_{CE}\left(\GCex_{(g),1}\right)$ is of order $6W-H$. 
\end{lemma}

In particular, if $g\geq 6W-H$ then the cohomology
$\gr^{(W,H)}H_{CE}\left(\GCex_{(g),1}\right)$ is completely determined by the requirement that 
\[
\left(\gr^{(W,H)}H_{CE}\left(\GCex_{(g),1}\right)\otimes V_g^M\otimes \left(V_g^*\right)^N\right)^{\GL_g}
=
H(C_{M,N}^\infty)
\]
for all $M,N\leq 6W-H$.

\subsection{Combinatorial form of $C_{M,N}^\infty$}

The auxiliary graph complex $C_{M,N}^\infty$ defined above has a combinatorial description as a graph complex, as can be seen from the proof of Lemma/Definition \ref{lemdef:Cinfty}.

Concretely, It has two different types of edges, solid undirected edges connecting internal vertices as well as dashed directed edges. The latter can start at internal or external vertices and point to internal, external or $\times$-vertices. Internal vertices have total valence at least three, external vertices have valence one with one incoming or one outgoing dashed edge, and $\times$-vertices have valence two with two incoming dashed edges.
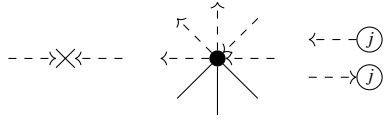
\[
\begin{tikzpicture}
  \node[int] (int) at (0,0) {};
  \draw (int) edge (-135: 0.75) (int) edge (-45: 0.75) (int) edge (-90: 0.75);
  \draw[->] (int) edge[dashed] (180: 0.75) (int) edge[dashed] (135: 0.75) (int) edge[dashed] (90: 0.75);
  \path[->] (45: 0.75) edge[dashed] (int);
  \path[->] (0: 0.75) edge[dashed] (int);
  \node[ext, minimum size=0.2cm] (ext1) at (2,0.25) {$\scriptstyle j$};
  \path[->] (ext1) edge[dashed] (1.2,0.25);
  \node[ext, minimum size=0.2cm] (ext2) at (2,-0.25) {$\scriptstyle j$};
  \path[->] (1.2,-0.25) edge[dashed] (ext2);
  \node[cross] (cross) at (-2,0) {}; 
  \path[->] (-1.25,0) edge[dashed] (cross);
  \path[->] (-2.75,0) edge[dashed] (cross);
  \end{tikzpicture}
\]
The following picture shows a graph of the auxiliary graph complex $C^\infty_{2,3}$.:
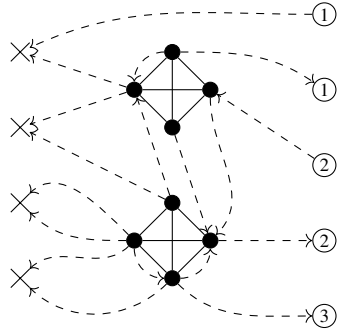
\[
\begin{tikzpicture}[baseline=-.65ex]
  \node[int] (v1) at (0, 1.5) {};
  \node[int] (v2) at (0.5, 1) {};
  \node[int] (v3) at (0, 0.5) {};
  \node[int] (v4) at (-0.5, 1) {};
  \draw (v1) edge (v2) edge (v3) edge (v4) (v2) edge (v3) edge (v4) (v3) edge (v4);
  \node[int] (v5) at (0, -0.5) {};
  \node[int] (v6) at (0.5, -1) {};
  \node[int] (v7) at (0, -1.5) {};
  \node[int] (v8) at (-0.5, -1) {};
  \draw (v5) edge (v6) edge (v7) edge (v8) (v6) edge (v7) edge (v8) (v7) edge (v8);
  \node[ext, minimum size=0.2cm] (e1) at (2, 2) {$\scriptstyle 1$};
  \node[ext, minimum size=0.2cm] (e2) at (2, 1) {$\scriptstyle 1$};
  \node[ext, minimum size=0.2cm] (e3) at (2, 0) {$\scriptstyle 2$};
  \node[ext, minimum size=0.2cm] (e4) at (2, -1) {$\scriptstyle 2$};
  \node[ext, minimum size=0.2cm] (e5) at (2, -2) {$\scriptstyle 3$};
  \node[cross] (x1) at (-2, 1.5) {};
  \node[cross] (x2) at (-2, 0.5) {};
  \node[cross] (x3) at (-2, -0.5) {};
  \node[cross] (x4) at (-2, -1.5) {};

  \draw[->] (v1) [out=0, in=145] edge[dashed] (e2);
  \draw[->] (e3) edge[dashed] (v2);
  \draw[->] (v7) [out=-45, in=180] edge[dashed] (e5);
  \draw[->] (v4) edge[dashed] (x1) (v4) edge[dashed] (x2) (v5) edge[dashed] (x2) (v8)[out=135, in=45] edge[dashed] (x3) (v8)[out=180, in=-45] edge[dashed] (x3);
  \draw[->](e1) [out=180, in=25] edge[dashed] (x1);
  \draw[->](v8) [out=-135, in=45]edge[dashed] (x4); 
  \draw[->](v7) [out=-135, in=-45] edge[dashed] (x4);
  \draw[->](v5) edge[dashed] (v4);
  \draw[->](v2) [out=-90, in=45] edge[dashed] (v6);
  \draw[->](v3) edge[dashed] (v6);
  \draw[->](v7) [out=0, in=-90] edge[dashed] (v6);
  \draw[->](v8) [out=-90, in=180] edge[dashed] (v7);
  \draw[->](v1) [out=180, in=90] edge[dashed] (v4);
  \draw[->] (v6) [] edge[dashed] (e4);
\end{tikzpicture}
\]

The natural map $C_{M,N}^\infty\to C_{M,N}^g$ has the following combinatorial description. 
Let $\Gamma\in C_{M,N}^\infty$ be a graph as above. It has edges of two types (solid or dashed), but no decorations in $V_g$ or $V_g^*$. 
Then we replace each dashed edge (say $(u,v)$, from vertex $u$ to vertex $v$) by one copy of the $\frac 12$-diagonal element 
\[
\Delta^{\frac 12} = (-1)^m
\sum_{i=1}^g \alpha_i \otimes \beta_i.
\]
More precisely, the first factor $\alpha_i$ will be multiiplied into the decoration of vertex $u$ and the second factor $\beta_i$ is multiplied into the decoration at vertex $v$.
Proceeding in this manner for all dashed edges, we obtain a linear combinations of graphs with no dashed edges but decorations in $V_g\oplus V_g^*$.
This linear combination is automatically $\GL_g$-invariant and hence defines an element of $C_{M,N}^g$, which is the image of the graph $\Gamma$.

The differential on $C_{M,N}^\infty$ has three components
\[
	d = d_{contract} + d_{cut} + d_{act}.
\]
The first two parts of the differential are identical to those on $F_{M,N}^\infty$ defined in section \ref{sec:combinatorial_F} and \ref{sec:combinatorial_G}. The third part of the differential, $d_{act}$, acts on all decorations of internal vertices by elements of $V_g$ with an element of $\osp_{g,1}^{nil}$ mapping it on $V_g^*$
\[
	\osp_{g,1}^{nil}\otimes V_g^* \rightarrow V_g \text{   or equivalently } V_g^* \rightarrow \left(\osp_{g,1}^{nil}\right)^* \otimes V_g,
\]
which corresponds in the graphical representation to summing over all dashed half edges $h$ incoming to an internal vertix and  extending it by a $\times$-vertex and and an outgoing dashed edge. 

Using the graphical notation the total differential of this graph complex is given as follows. For the first two parts, $d_{contract} + d_{cut}$, we sum over all solid undirected edges $e$ and replace it by three terms. We contract the edge $e$ and connect all incident edges of the two affected vertices to the one remaining vertex. Second, we replace the solid undirected edge $e$ by directed dashed edges in ether direction, whereat the latter two terms have opposite sign.
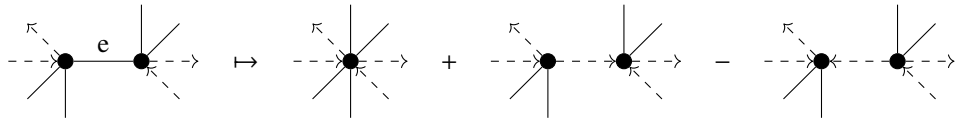
\[
	\begin{tikzpicture}
  		\node[int] (v1) at (0,0) {};
		\node[int] (v2) at (1,0) {};
		\draw (v1) edge node[midway, above]{e} (v2);
		\path[->] (-0.75,0) edge[dashed] (v1);
		\path[->] (v1) edge[dashed] (-0.5,0.5);
		\draw (-0.5,-0.5) edge (v1);
		\draw (0,-0.75) edge (v1);
		\path[->] (1.5,-0.5) edge[dashed] (v2);
		\path[->] (v2) edge[dashed] (1.75,0);
		\draw (1, 0.75) edge (v2);
		\draw (1.5, 0.5) edge (v2);
  	\end{tikzpicture}
  	\quad\mapsto\quad
  	\begin{tikzpicture}
  		\node[int] (v1) at (0,0) {};
		\path[->] (-0.75,0) edge[dashed] (v1);
		\path[->] (v1) edge[dashed] (-0.5,0.5);
		\draw (-0.5,-0.5) edge (v1);
		\draw (0,-0.75) edge (v1);
		\path[->] (0.5,-0.5) edge[dashed] (v1);
		\path[->] (v1) edge[dashed] (0.75,0);
		\draw (0, 0.75) edge (v1);
		\draw (0.5, 0.5) edge (v1);
  	\end{tikzpicture}
  	\quad+\quad
	\begin{tikzpicture}
  		\node[int] (v1) at (0,0) {};
		\node[int] (v2) at (1,0) {};
		\draw[->] (v1) edge[dashed] (v2);
		\path[->] (-0.75,0) edge[dashed] (v1);
		\path[->] (v1) edge[dashed] (-0.5,0.5);
		\draw (-0.5,-0.5) edge (v1);
		\draw (0,-0.75) edge (v1);
		\path[->] (1.5,-0.5) edge[dashed] (v2);
		\path[->] (v2) edge[dashed] (1.75,0);
		\draw (1, 0.75) edge (v2);
		\draw (1.5, 0.5) edge (v2);
  	\end{tikzpicture}
  	\quad-\quad
  	\begin{tikzpicture}
  		\node[int] (v1) at (0,0) {};
		\node[int] (v2) at (1,0) {};
		\draw[->] (v2) edge[dashed] (v1);
		\path[->] (-0.75,0) edge[dashed] (v1);
		\path[->] (v1) edge[dashed] (-0.5,0.5);
		\draw (-0.5,-0.5) edge (v1);
		\draw (0,-0.75) edge (v1);
		\path[->] (1.5,-0.5) edge[dashed] (v2);
		\path[->] (v2) edge[dashed] (1.75,0);
		\draw (1, 0.75) edge (v2);
		\draw (1.5, 0.5) edge (v2);
  	\end{tikzpicture}
\]
For the third part $d_{act}$ of the differential we sum over all half edges h of directed dashed edges pointing to an internal vertex and extend it with an outgoing dashed edge and a $\times$-vertex.
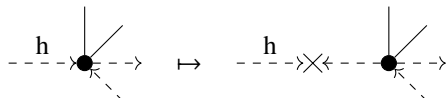
\[
	\begin{tikzpicture}
  		\node[int] (v2) at (0,0) {};
		\path[->] (-1,0) edge[dashed] node[midway, above]{h} (v2);
		\path[->] (0.5,-0.5) edge[dashed] (v2);
		\path[->] (v2) edge[dashed] (0.75,0);
		\draw (0, 0.75) edge (v2);
		\draw (0.5, 0.5) edge (v2);
  	\end{tikzpicture}
  	\quad\mapsto\quad
  		\begin{tikzpicture}
  		\node[int] (v2) at (0,0) {};
  		\node[cross] (x1) at (-1,0){};
		\path[->] (-2,0) edge[dashed] node[midway, above]{h} (x1);
		\path[->] (v2) edge[dashed] (x1);
		\path[->] (0.5,-0.5) edge[dashed] (v2);
		\path[->] (v2) edge[dashed] (0.75,0);
		\draw (0, 0.75) edge (v2);
		\draw (0.5, 0.5) edge (v2);
  	\end{tikzpicture}
\]

\subsection{Transformed auxiliary graph complex $\tilde{C}_{M,N}^\infty$}
We define again a new transformed auxiliary graph complex $\tilde{C}_{M,N}^\infty$ in the same way as  the auxiliary graph complex $\tilde{G}_{M,N}^\infty$ in section \ref{sec:auxiliary_G}. 


In order to define the differential on the new transformed graph complex we sum over all solid edges $e$, all dashed half edges h pointing to an internal vertex as well as all dashed $\oplus$- and $\ominus$-edges. The differential then has the following terms.
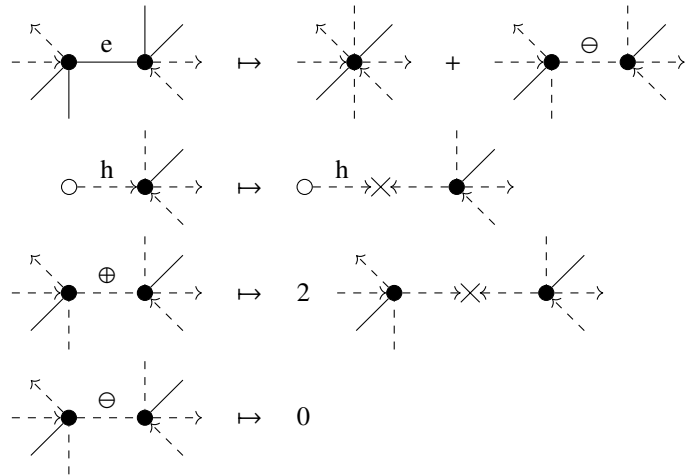
\[
\label{diff_aux_complex}
\begin{aligned}
	\begin{tikzpicture}
  		\node[int] (v1) at (0,0) {};
		\node[int] (v2) at (1,0) {};
		\draw (v1) edge node[midway, above]{e} (v2);
		\path[->] (-0.75,0) edge[dashed] (v1);
		\path[->] (v1) edge[dashed] (-0.5,0.5);
		\draw (-0.5,-0.5) edge (v1);
		\draw (0,-0.75) edge (v1);
		\path[->] (1.5,-0.5) edge[dashed] (v2);
		\path[->] (v2) edge[dashed] (1.75,0);
		\draw (1, 0.75) edge (v2);
		\draw (1.5, 0.5) edge (v2);
  	\end{tikzpicture}
  	\quad &\mapsto\quad
  	\begin{tikzpicture}
  		\node[int] (v1) at (0,0) {};
		\path[->] (-0.75,0) edge[dashed] (v1);
		\path[->] (v1) edge[dashed] (-0.5,0.5);
		\draw (-0.5,-0.5) edge (v1);
		\draw (0,-0.75) edge[dashed] (v1);
		\path[->] (0.5,-0.5) edge[dashed] (v1);
		\path[->] (v1) edge[dashed] (0.75,0);
		\draw (0, 0.75) edge[dashed] (v1);
		\draw (0.5, 0.5) edge (v1);
  	\end{tikzpicture}
  	\quad+\quad
	\begin{tikzpicture}
  		\node[int] (v1) at (0,0) {};
		\node[int] (v2) at (1,0) {};
		\draw (v1) edge[dashed] node[midway, above]{$\ominus$} (v2);
		\path[->] (-0.75,0) edge[dashed] (v1);
		\path[->] (v1) edge[dashed] (-0.5,0.5);
		\draw (-0.5,-0.5) edge (v1);
		\draw (0,-0.75) edge[dashed] (v1);
		\path[->] (1.5,-0.5) edge[dashed] (v2);
		\path[->] (v2) edge[dashed] (1.75,0);
		\draw (1, 0.75) edge[dashed] (v2);
		\draw (1.5, 0.5) edge (v2);
  	\end{tikzpicture}\\
	\begin{tikzpicture}
  		\node[int] (v2) at (0,0) {};
  		\node[ext, minimum size=0.2cm] (e1) at (-1, 0) {};
		\path[->] (e1) edge[dashed] node[midway, above]{h} (v2);
		\path[->] (0.5,-0.5) edge[dashed] (v2);
		\path[->] (v2) edge[dashed] (0.75,0);
		\draw (0, 0.75) edge[dashed] (v2);
		\draw (0.5, 0.5) edge (v2);
  	\end{tikzpicture}
  	\quad &\mapsto\quad
  	\begin{tikzpicture}
  		\node[int] (v2) at (0,0) {};
  		\node[cross] (x1) at (-1,0) {};
  		\node[ext, minimum size=0.2cm] (e1) at (-2, 0) {};
		\path[->] (e1) edge[dashed] node[midway, above]{h} (x1);
		\path[->] (v2) edge[dashed] (x1);
		\path[->] (0.5,-0.5) edge[dashed] (v2);
		\path[->] (v2) edge[dashed] (0.75,0);
		\draw (0, 0.75) edge[dashed] (v2);
		\draw (0.5, 0.5) edge (v2);
  	\end{tikzpicture}\\
	\begin{tikzpicture}
  		\node[int] (v1) at (0,0) {};
		\node[int] (v2) at (1,0) {};
		\draw (v1) edge[dashed] node[midway, above]{$\oplus$} (v2);
		\path[->] (-0.75,0) edge[dashed] (v1);
		\path[->] (v1) edge[dashed] (-0.5,0.5);
		\draw (-0.5,-0.5) edge (v1);
		\draw (0,-0.75) edge[dashed] (v1);
		\path[->] (1.5,-0.5) edge[dashed] (v2);
		\path[->] (v2) edge[dashed] (1.75,0);
		\draw (1, 0.75) edge[dashed] (v2);
		\draw (1.5, 0.5) edge (v2);
  	\end{tikzpicture}
  	\quad &\mapsto\quad2\quad
  	\begin{tikzpicture}
  		\node[int] (v1) at (-1,0) {};
		\node[int] (v2) at (1,0) {};
		\node[cross] (x1) at (0,0) {};
		\draw[->] (v1) edge[dashed] (x1) (v2) edge[dashed] (x1);
		\path[->] (-1.75,0) edge[dashed] (v1);
		\path[->] (v1) edge[dashed] (-1.5,0.5);
		\draw (-1.5,-0.5) edge (v1);
		\draw (-1,-0.75) edge[dashed] (v1);
		\path[->] (1.5,-0.5) edge[dashed] (v2);
		\path[->] (v2) edge[dashed] (1.75,0);
		\draw (1, 0.75) edge[dashed] (v2);
		\draw (1.5, 0.5) edge (v2);
  	\end{tikzpicture}\\
	\begin{tikzpicture}
  		\node[int] (v1) at (0,0) {};
		\node[int] (v2) at (1,0) {};
		\draw (v1) edge[dashed] node[midway, above]{$\ominus$} (v2);
		\path[->] (-0.75,0) edge[dashed] (v1);
		\path[->] (v1) edge[dashed] (-0.5,0.5);
		\draw (-0.5,-0.5) edge (v1);
		\draw (0,-0.75) edge[dashed] (v1);
		\path[->] (1.5,-0.5) edge[dashed] (v2);
		\path[->] (v2) edge[dashed] (1.75,0);
		\draw (1, 0.75) edge[dashed] (v2);
		\draw (1.5, 0.5) edge (v2);
  	\end{tikzpicture}
  	\quad &\mapsto\quad 0
\end{aligned}
\]

\subsection{Cohomology of the auxiliary graph complex $\tilde{C}_{M,N}^\infty$}

Define the the subspace $R_{M,N}\subset \tilde C_{M,N}^\infty$ spanned by graphs that have no edges connecting two non-external vertices, and all edges from internal to external vertices are outgoing.
In other words, permissible graphs are unions of the following types of connected components:


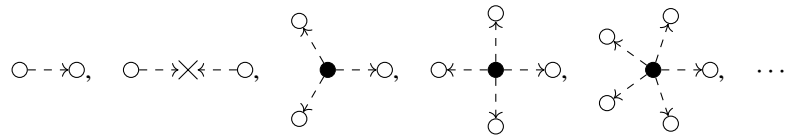
\[
	\begin{tikzpicture}
		\node[ext, minimum size=0.2cm] (e1) at (-0.75,0) {};
		\node[ext, minimum size=0.2cm] (e2) at (0,0) {};
		\path[->] (e1) edge[dashed] (e2);
  	\end{tikzpicture}
  	,\quad
  	\begin{tikzpicture}
		\node[ext, minimum size=0.2cm] (e1) at (-0.75,0) {};
		\node[ext, minimum size=0.2cm] (e2) at (0.75,0) {};
		\node[cross] (x1) at (0,0) {};
		\path[->] (e1) edge[dashed] (x1);
		\path[->] (e2) edge[dashed] (x1);
  	\end{tikzpicture}
  	,\quad
	\begin{tikzpicture}
		\node[int] (v1) at (0,0) {};
		\node[ext, minimum size=0.2cm] (e1) at (0:0.75) {};
		\node[ext, minimum size=0.2cm] (e2) at (120:0.75) {};
		\node[ext, minimum size=0.2cm] (e3) at (-120:0.75) {};
		\path[->] (v1) edge[dashed] (e1) (v1) edge[dashed] (e2) (v1) edge[dashed] (e3) ;
  	\end{tikzpicture}
  	,\quad
	\begin{tikzpicture}
		\node[int] (v1) at (0,0) {};
		\node[ext, minimum size=0.2cm] (e1) at (0:0.75) {};
		\node[ext, minimum size=0.2cm] (e2) at (90:0.75) {};
		\node[ext, minimum size=0.2cm] (e3) at (180:0.75) {};
		\node[ext, minimum size=0.2cm] (e4) at (-90:0.75) {};
		\path[->] (v1) edge[dashed] (e1) (v1) edge[dashed] (e2) (v1) edge[dashed] (e3) (v1) edge[dashed] (e4) ;
  	\end{tikzpicture}
  	,\quad
  	\begin{tikzpicture}
		\node[int] (v1) at (0,0) {};
		\node[ext, minimum size=0.2cm] (e1) at (0:0.75) {};
		\node[ext, minimum size=0.2cm] (e2) at (72:0.75) {};
		\node[ext, minimum size=0.2cm] (e3) at (144:0.75) {};
		\node[ext, minimum size=0.2cm] (e4) at (216:0.75) {};
		\node[ext, minimum size=0.2cm] (e5) at (288:0.75) {};
		\path[->] (v1) edge[dashed] (e1) (v1) edge[dashed] (e2) (v1) edge[dashed] (e3) (v1) edge[dashed] (e4) (v1) edge[dashed] (e5);
  	\end{tikzpicture}
  	,\quad \cdots
\]
with the numbering at external vertices omitted. It is clear that all graphs as above are cocycles, and hence $R_{M,N}$ is a subcomplex, equipped with zero differential.

\begin{prop}\label{prop:chevgcex}
	The inclusion $R_{M,N} \hookrightarrow \tilde{C}_{M,N}^\infty \cong C_{M,N}^\infty$ is a quasi-isomorphism.  Hence, the cohomology of the auxiliary graph complex is isomorphic to $R_{M,N}$.

\end{prop}

\begin{proof}
 
In a first step we apply Lemma \ref{lemma:inclusion} with $V=\tilde{C}_{M,N}^\infty$ and $U=R_{M,N}$.
Therefore, we are left to show that $\tilde{C}_{M,N}^\infty/R_{M,N}$ is acyclic. Let us consider the bounded filtration of $\tilde{C}_{M,N}^\infty/R_{M,N}$ with respect to the number of internal vertices $v$ and consider the associated graded complex $gr^v\tilde{C}_{M,N}^\infty/R_{M,N}$. From Lemma \ref{lemma:graded} it can be deduced that if the associate graded complex $gr^v\tilde{C}_{M,N}^\infty/R_{M,N}$ is acyclic then so is $\tilde{C}_{M,N}^\infty/R_{M,N}$. 

Finally, we have to show that the associated graded complex $gr^v\tilde{C}_{M,N}^\infty/R_{M,N}$ is acyclic. The differential $d^v$ of this graded complex reads as follows. We sum again over all solid edges $e$, all half edges $h$ pointing towards an inner vertex as well as over all $\oplus$-edges.
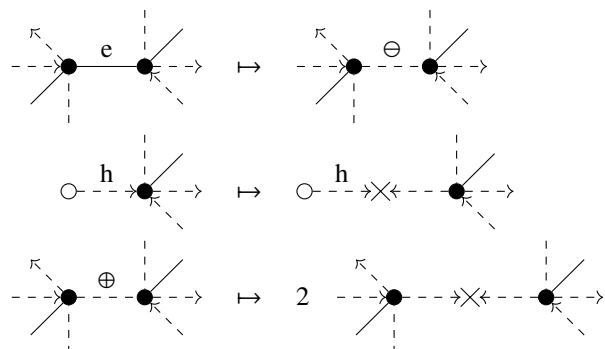
\[
\begin{aligned}
	\begin{tikzpicture}
  		\node[int] (v1) at (0,0) {};
		\node[int] (v2) at (1,0) {};
		\draw (v1) edge node[midway, above]{e} (v2);
		\path[->] (-0.75,0) edge[dashed] (v1);
		\path[->] (v1) edge[dashed] (-0.5,0.5);
		\draw (-0.5,-0.5) edge (v1);
		\draw (0,-0.75) edge[dashed] (v1);
		\path[->] (1.5,-0.5) edge[dashed] (v2);
		\path[->] (v2) edge[dashed] (1.75,0);
		\draw (1, 0.75) edge[dashed] (v2);
		\draw (1.5, 0.5) edge (v2);
  	\end{tikzpicture}
  	\quad &\mapsto\quad
	\begin{tikzpicture}
  		\node[int] (v1) at (0,0) {};
		\node[int] (v2) at (1,0) {};
		\draw (v1) edge[dashed] node[midway, above]{$\ominus$} (v2);
		\path[->] (-0.75,0) edge[dashed] (v1);
		\path[->] (v1) edge[dashed] (-0.5,0.5);
		\draw (-0.5,-0.5) edge (v1);
		\draw (0,-0.75) edge[dashed] (v1);
		\path[->] (1.5,-0.5) edge[dashed] (v2);
		\path[->] (v2) edge[dashed] (1.75,0);
		\draw (1, 0.75) edge[dashed] (v2);
		\draw (1.5, 0.5) edge (v2);
  	\end{tikzpicture}\\
		\begin{tikzpicture}
  		\node[int] (v2) at (0,0) {};
  		\node[ext, minimum size=0.2cm] (e1) at (-1, 0) {};
		\path[->] (e1) edge[dashed] node[midway, above]{h} (v2);
		\path[->] (0.5,-0.5) edge[dashed] (v2);
		\path[->] (v2) edge[dashed] (0.75,0);
		\draw (0, 0.75) edge[dashed] (v2);
		\draw (0.5, 0.5) edge (v2);
  	\end{tikzpicture}
  	\quad &\mapsto\quad
  	\begin{tikzpicture}
  		\node[int] (v2) at (0,0) {};
  		\node[cross] (x1) at (-1,0) {};
  		\node[ext, minimum size=0.2cm] (e1) at (-2, 0) {};
		\path[->] (e1) edge[dashed] node[midway, above]{h} (x1);
		\path[->] (v2) edge[dashed] (x1);
		\path[->] (0.5,-0.5) edge[dashed] (v2);
		\path[->] (v2) edge[dashed] (0.75,0);
		\draw (0, 0.75) edge[dashed] (v2);
		\draw (0.5, 0.5) edge (v2);
  	\end{tikzpicture}\\
  	\begin{tikzpicture}
  		\node[int] (v1) at (0,0) {};
		\node[int] (v2) at (1,0) {};
		\draw (v1) edge[dashed] node[midway, above]{$\oplus$} (v2);
		\path[->] (-0.75,0) edge[dashed] (v1);
		\path[->] (v1) edge[dashed] (-0.5,0.5);
		\draw (-0.5,-0.5) edge (v1);
		\draw (0,-0.75) edge[dashed] (v1);
		\path[->] (1.5,-0.5) edge[dashed] (v2);
		\path[->] (v2) edge[dashed] (1.75,0);
		\draw (1, 0.75) edge[dashed] (v2);
		\draw (1.5, 0.5) edge (v2);
  	\end{tikzpicture}
  	\quad &\mapsto\quad2\quad
  	\begin{tikzpicture}
  		\node[int] (v1) at (-1,0) {};
		\node[int] (v2) at (1,0) {};
		\node[cross] (x1) at (0,0) {};
		\draw[->] (v1) edge[dashed] (x1) (v2) edge[dashed] (x1);
		\path[->] (-1.75,0) edge[dashed] (v1);
		\path[->] (v1) edge[dashed] (-1.5,0.5);
		\draw (-1.5,-0.5) edge (v1);
		\draw (-1,-0.75) edge[dashed] (v1);
		\path[->] (1.5,-0.5) edge[dashed] (v2);
		\path[->] (v2) edge[dashed] (1.75,0);
		\draw (1, 0.75) edge[dashed] (v2);
		\draw (1.5, 0.5) edge (v2);
  	\end{tikzpicture}\\
\end{aligned}
\]
We can now define a homotopy $h^v$ corresponding to the inverse of this differential.
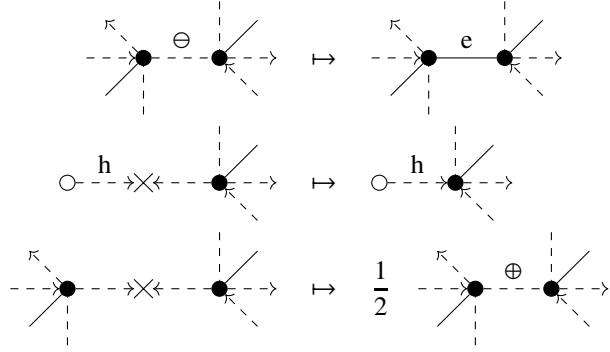
\[
\begin{aligned}
	\begin{tikzpicture}
  		\node[int] (v1) at (0,0) {};
		\node[int] (v2) at (1,0) {};
		\draw (v1) edge[dashed] node[midway, above]{$\ominus$} (v2);
		\path[->] (-0.75,0) edge[dashed] (v1);
		\path[->] (v1) edge[dashed] (-0.5,0.5);
		\draw (-0.5,-0.5) edge (v1);
		\draw (0,-0.75) edge[dashed] (v1);
		\path[->] (1.5,-0.5) edge[dashed] (v2);
		\path[->] (v2) edge[dashed] (1.75,0);
		\draw (1, 0.75) edge[dashed] (v2);
		\draw (1.5, 0.5) edge (v2);
  	\end{tikzpicture}
  	\quad &\mapsto\quad
	\begin{tikzpicture}
  		\node[int] (v1) at (0,0) {};
		\node[int] (v2) at (1,0) {};
		\draw (v1) edge node[midway, above]{e} (v2);
		\path[->] (-0.75,0) edge[dashed] (v1);
		\path[->] (v1) edge[dashed] (-0.5,0.5);
		\draw (-0.5,-0.5) edge (v1);
		\draw (0,-0.75) edge[dashed] (v1);
		\path[->] (1.5,-0.5) edge[dashed] (v2);
		\path[->] (v2) edge[dashed] (1.75,0);
		\draw (1, 0.75) edge[dashed] (v2);
		\draw (1.5, 0.5) edge (v2);
  	\end{tikzpicture}\\
  	\begin{tikzpicture}
  		\node[int] (v2) at (0,0) {};
  		\node[cross] (x1) at (-1,0) {};
  		\node[ext, minimum size=0.2cm] (e1) at (-2, 0) {};
		\path[->] (e1) edge[dashed] node[midway, above]{h} (x1);
		\path[->] (v2) edge[dashed] (x1);
		\path[->] (0.5,-0.5) edge[dashed] (v2);
		\path[->] (v2) edge[dashed] (0.75,0);
		\draw (0, 0.75) edge[dashed] (v2);
		\draw (0.5, 0.5) edge (v2);
  	\end{tikzpicture}
  	\quad &\mapsto\quad
  	\begin{tikzpicture}
  		\node[int] (v2) at (0,0) {};
  		\node[ext, minimum size=0.2cm] (e1) at (-1, 0) {};
		\path[->] (e1) edge[dashed] node[midway, above]{h} (v2);
		\path[->] (0.5,-0.5) edge[dashed] (v2);
		\path[->] (v2) edge[dashed] (0.75,0);
		\draw (0, 0.75) edge[dashed] (v2);
		\draw (0.5, 0.5) edge (v2);
  	\end{tikzpicture}\\
  	\begin{tikzpicture}
  		\node[int] (v1) at (-1,0) {};
		\node[int] (v2) at (1,0) {};
		\node[cross] (x1) at (0,0) {};
		\draw[->] (v1) edge[dashed] (x1) (v2) edge[dashed] (x1);
		\path[->] (-1.75,0) edge[dashed] (v1);
		\path[->] (v1) edge[dashed] (-1.5,0.5);
		\draw (-1.5,-0.5) edge (v1);
		\draw (-1,-0.75) edge[dashed] (v1);
		\path[->] (1.5,-0.5) edge[dashed] (v2);
		\path[->] (v2) edge[dashed] (1.75,0);
		\draw (1, 0.75) edge[dashed] (v2);
		\draw (1.5, 0.5) edge (v2);
  	\end{tikzpicture}
  	\quad &\mapsto\quad\frac{1}{2}\quad
  	\begin{tikzpicture}
  		\node[int] (v1) at (0,0) {};
		\node[int] (v2) at (1,0) {};
		\draw (v1) edge[dashed] node[midway, above]{$\oplus$} (v2);
		\path[->] (-0.75,0) edge[dashed] (v1);
		\path[->] (v1) edge[dashed] (-0.5,0.5);
		\draw (-0.5,-0.5) edge (v1);
		\draw (0,-0.75) edge[dashed] (v1);
		\path[->] (1.5,-0.5) edge[dashed] (v2);
		\path[->] (v2) edge[dashed] (1.75,0);
		\draw (1, 0.75) edge[dashed] (v2);
		\draw (1.5, 0.5) edge (v2);
  	\end{tikzpicture}\\
\end{aligned}
\]

Finally, we apply Lemma \ref{lem:homotopy} with $V=gr^v\tilde{C}_{M,N}^\infty/R_{M,N}$ and $A=d^vh^v+h^vd^v=\mathds{1}$.
Hence, the associate graded complex $gr^v\tilde{C}_{M,N}^\infty/R_{M,N}$ is acyclic, which finishes the proof.
\end{proof}


By Lemma \ref{lem:HCEW} the Chevalley-Eilenberg cohomology $\gr^{(W,H)}H_{CE}\left(\GCex_{(g),1}\right)$ may be determined from the previous Proposition as long as $g\geq 6W-H$.
However, explicit formulas are not easy to obtain.
We shall hence restrict to extracting a set of generators.
To this end consider the graphs with a single internal vertex decorated by an element of $S^{\geq3}(H^m(W_{g,1})[-m])$
\[
	T:=\left\langle
	\begin{tikzpicture}
  		\node[int] at (0,0) {};
  		\node at (0,0.3) {$U$};
  	\end{tikzpicture}
  	\right\rangle
  	\quad\text{with}\quad
  	U\in S^{\geq3}(H^m(W_{g,1})[-m]).
\]


From Proposition \ref{prop:chevgcex} we then obtain in particular the following Corollary.
\begin{cor}
\label{cor:chevgcex}	
The map 
$\gr^{(W,H)}\left(S\left(\left(\osp_{g,1}^{nil}\right)^*[-1]\right)\otimes S\left(T[-1]\right)\right) \to \gr^{(W,H)} H\left(C_{CE}\left(\GCex_{(g),1}\right)\right)$
is surjective as long as $g\geq 6W-H$.
\end{cor}

\end{document}